\documentclass[reqno, 11pt]{amsart}

\usepackage[algoruled,lined,noresetcount,norelsize]{algorithm2e}

\usepackage{graphicx}
\usepackage{amsmath,amssymb,amsthm}
\usepackage{mathtools}
\usepackage{mathrsfs}
\usepackage{mathabx}\changenotsign
\usepackage{dsfont}
\usepackage{xcolor}
\usepackage[backref]{hyperref}
\hypersetup{
    colorlinks,
    linkcolor={red!60!black},
    citecolor={green!60!black},
    urlcolor={blue!60!black}
}
\usepackage[T1]{fontenc}
\usepackage{lmodern}
\usepackage[babel]{microtype}
\usepackage[english]{babel}

\usepackage{centernot}

\usepackage{geometry}
\numberwithin{equation}{section}
\numberwithin{figure}{section}

\usepackage{enumitem}

\theoremstyle{plain}
\newtheorem{thm}{Theorem}[section]
\newtheorem{prop}[thm]{Proposition}

\newtheorem{clm}[thm]{Claim}

\newtheorem{lemma}[thm]{Lemma}
\newtheorem{definition}[thm]{Definition}
\newtheorem{problem}[thm]{Problem}
\newtheorem{conjecture}[thm]{Conjecture}

\newcommand{\MBarrow}{\xrightarrow{MB}}
\newcommand{\MBnotarrow}{\centernot{\xrightarrow{MB}}}
\newcommand{\BMarrow}{\xrightarrow{BM}}
\newcommand{\MBequiv}{\stackrel{\text{\tiny MB}}{\sim}}

\title{Ramsey-type results for Maker-Breaker games}

\author[A. Allin]{Alexander Allin}
\author[J. Barkey]{Juri Barkey}
\author[D. Clemens]{Dennis Clemens}

\address{Hamburg University of Technology, Institute of Mathematics, Am Schwarzenberg-Campus 3, 21073 Hamburg, Germany}
\email{alexander.allin|juri.barkey|dennis.clemens@tuhh.de}

\begin{document}

\begin{abstract}
A graph $G$ is minimal Ramsey for a graph $H$ if every $2$-colouring of the edges of $G$ contains a monochromatic copy of $H$, but for every proper subgraph of $G$, there is a $2$-colouring that does not contain such a monochromatic copy. Characterizing minimal Ramsey graphs is a widely studied problem. Recent research in this field includes the characterization of the size of the set $\mathcal{M}_2(H)$ of all minimal Ramsey graphs for $H$, or finding the smallest minimum degree among all graphs in $\mathcal{M}_2(H)$.

In this paper, we introduce a game theoretic analogue of the above concept by considering the Maker-Breaker $H$-game on a graph $G$. In this game, two players, Maker and Breaker, alternately claim unclaimed edges of $G$, and Maker wins if in the end of the game the graph spanned by Maker's edges contains a copy of $H$. Otherwise, Breaker wins the game.
We call a graph $G$ winnable for $H$ if Maker has a winning strategy for the Maker-Breaker $H$-game on $G$, and we call it minimal winnable if additionally Breaker wins the $H$-game on every proper subgraph of $G$.
Along the lines of minimal-Ramsey theory, we 
characterize all graphs $H$ for which there exist infinitely many minimal winnable graphs, and we prove tight bounds for the smallest minimum degree among all these minimal winnable graphs. Amongst others, we obtain precise results for trees, cycles, cliques, and complete bipartite graphs. In general, we find many similarities between the Ramsey setting and the Maker-Breaker setting, but we also show substantial differences. 
\end{abstract}

\maketitle

\begin{center}
\end{center}

\section{Introduction}

Given a positive integer $b$ and a hypergraph $(X,\mathcal{F})$, 
the $(1:b)$ \emph{Maker-Breaker game} on the hypergraph $(X,\mathcal{F})$ is played as follows.
Maker and Breaker claim unclaimed elements of the \emph{board} $X$ alternatingly,
where in every round, Maker is always allowed to claim $1$ element and Breaker is allowed to claim up to $b$ elements. If until the end of the game, Maker manages to claim all elements of a \emph{winning set} $F\in \mathcal{F}$, she wins the game. Otherwise, Breaker wins.

The research on Maker-Breaker games has a long history already, starting from fundamental papers
of e.g.~Chv\'atal and Erd\H{o}s~\cite{chvatal1978biased}, or Erd\H{o}s and Selfridge~\cite{erdos1973combinatorial}, and
in many cases these games have proven to have intriguing connections to extremal combinatorics, Ramsey theory or random graph theory; see 
e.g.~the books~\cite{beck2008combinatorial,hefetz2014positional}.
The focus of this paper is to find further connections to and comparisons with Ramsey-type results.
More precisely, along the line of research on Ramsey graphs we intend to study the structure of graphs
on which so-called Maker-Breaker $H$-games (to be defined later) can be won.

\subsection{Some background on Ramsey theory}

Given graphs $G$ and $H$ and a number $q$ of colours, we say that 
$G$ is \emph{$q$-Ramsey} for $H$, denoted by $G\rightarrow_q H$, if  the following holds:
in every colouring of the edges of $G$ with $q$ colours, we find a copy of $H$
that is monochromatic. Let us denote with $\mathcal{R}_q(H)$ the set of all $q$-Ramsey graphs for $H$.
Then the fundamental theorem of Ramsey~\cite{ramsey1987problem}
ensures that $\mathcal{R}_q(H) \neq \varnothing$ holds for every $q$ and $H$. 
Having thus shown the existence of Ramsey graphs,
it becomes a natural problem to better understand the structure of graphs within the set $\mathcal{R}_q(H)$.

For instance, one may ask how large a graph in $\mathcal{R}_q(H)$ needs to be at least.
More formally, given a graph $H$ and a positive integer $q$,
we define the \emph{$q$-Ramsey number} $r_q(H)$ to be the smallest integer $n$ such that there exists a complete graph $K_n$ with $K_n \rightarrow_q H$. 
A plethora of papers deals with such Ramsey numbers for various graphs $H$; for an overview see e.g.~\cite{conlon2015recent,xu2018ramsey}.
Probably the most interesting and challenging case is when $H$ is a complete graph. When $q=2$, it is already an open problem to determine $r_2(K_5)$ exactly,
while in general the best known bounds are of the form $(1-o(1)) \frac{\sqrt{2}}{e} t2^{t/2} \leq r_2(K_t) \leq 3.8^{t}$.
These bounds are due to Spencer~\cite{spencer1975ramsey} and Gupta, Ndiaye, Norin and Wei~\cite{gupta2024optimizing},
improving on a sequence of earlier results~\cite{conlon2009new,erdos1947some,erdos1935combinatorial,sah2023diagonal,thomason1988upper}, where the upper bound builds on the breakthrough paper of Campos, Griffiths, Morris and Shahasrabudhe~\cite{campos2026exponential}.

Next to minimizing the number of vertices of a Ramsey graph, one may also consider other graph parameters.
For instance, a surprising result of Folkman~\cite{folkman1970graphs} gives that for every positive integer $n$ there exists
a graph $G$ with clique number $\omega(G)=n$ such that $G\rightarrow_2 K_n$. This was later generalized to arbitrary graphs $H$ and arbitrary numbers of colours by Ne\v{s}et\v{r}il and Rödl~\cite{nevsetvril1976ramsey}.

\begin{thm}[Theorem 1 in~\cite{nevsetvril1976ramsey}] \label{thm:Ramsey:clique.number}
For every graph $H$ and every positive integer $q$
there exists a graph $G$ with $\omega(G)=\omega(H)$
such that $G\rightarrow_q H$.
\end{thm}

A more systematic study of graph parameters for Ramsey graphs was then initiated by Burr, Erd\H{o}s and Lov\'asz~\cite{burr1976graphs} who looked at e.g.~the minimum degree, the chromatic number and the vertex-connectivity of Ramsey graphs for the complete graph $K_t$, when the number of colours is $q=2$. For this however note that, when we want to minimize or maximize these parameters, in some cases we obtain trivial results. In fact, if $G\rightarrow_2 K_t$ holds and $G^\ast$ is obtained from $G$ by adding an isolated vertex, then $G^\ast\rightarrow_2 K_t$ is true as well. Hence, we easily find a $2$-Ramsey graph for $K_t$ of minimum degree $0$, which is clearly the smallest value possible. To avoid such trivialities, Burr, Erd\H{o}s and Lov\'asz~\cite{burr1976graphs} only studied those Ramsey graphs which are minimal with respect to subgraph containment. More formally, we say that $G$ is a \emph{minimal $q$-Ramsey graph} for $H$, if $G\rightarrow_q H$ holds and additionally $G' \not\rightarrow_q H$ is true for every proper subgraph $G'$ of $G$. 
Moreover, we let $\mathcal{M}_q(H)$ denote the set of all minimal $q$-Ramsey graphs for $H$.

Burr, Erd\H{o}s and Lov\'asz~\cite{burr1976graphs} obtained the following results: If we let
$s_2(K_t)$, $r_c(K_t)$ and $r_{\kappa}(K_t)$ denote the smallest possible
minimum degree, chromatic number and vertex-connectivity among all graphs in $\mathcal{M}_2(K_t)$,
then $s_2(K_t) = (t-1)^2$, $r_c(K_t) = r_2(K_t)$ and $r_{\kappa}(K_t) = 3$. 
The result on the minimum degree is very surprising for two reasons. First of all, it is a precise result, which rarely happens in Ramsey theory. Secondly, the value $s_2(K_t)$ is much smaller than the Ramsey number $r_2(K_t)$.
While every $2$-Ramsey graph for $K_t$ needs to have exponentially many vertices,
we can find such graphs that have a vertex of degree $(t-1)^2$ whose deletion destroys the Ramsey property.

Finally, the paper by Burr, Erd\H{o}s and Lov\'asz lead to further research on minimal Ramsey graphs.
The papers~\cite{burr1981ramsey,burr1978class,nevsetvril1978structure,rodl1995threshold} 
study the question when a graph $H$ possesses infinitely many minimal Ramsey graphs.
The following characterization can be obtained by combining their results.

\begin{thm}[Theorem follows from~\cite{burr1981ramsey,burr1978class,nevsetvril1978structure,rodl1995threshold}] \label{thm:Ramsey.finite}
The set $\mathcal{M}_2(H)$ is finite if and only if 
$H$ is the disjoint union of an odd star and a matching.
\end{thm}

In the subsequent years a big focus was put on generalizing the results of Burr, Erd\H{o}s and Lov\'asz
to general graphs and arbitrary numbers of colours. The parameter which received most attention
is the minimal minimum degree of minimal Ramsey graphs, formally
$$
s_q(H) := \min\{\delta(G):~ G\in \mathcal{M}_q(H)\} .
$$
In general the following bounds hold for this parameter. 
A proof of this statement is given in~\cite{fox2007minimum} for $q=2$ colours,
and it can be generalized easily to more than 2 colours.

\begin{thm}[Generalization of Theorem 3 in~\cite{fox2007minimum}]\label{thm:Ramsey.degree.bounds}
For every graph $H$ and every positive integer $q$, it holds that
$$q(\delta(H)-1)+1 \leq s_q(H) \leq r_q(H) - 1 . $$
\end{thm}

Szab\'o, Zumstein and Zürcher~\cite{szabo2010minimum} proved that for $q=2$ the lower bound 
in Theorem~\ref{thm:Ramsey.degree.bounds}
is tight for various bipartite graphs, including all trees, even cycles and complete bipartite graphs,
and they conjectured that it is tight for all bipartite graphs $H$.
Boyadzhiyska, Clemens and Gupta~\cite{boyadzhiyska2022minimal} showed that the lower bound is also
tight for arbitrary cycles and arbitrary $q$.
Moreover, Boyadzhiyska et al~\cite{boyadzhiyska2025ramsey} showed that for an Erd\H{o}s-R\'enyi random graph $G\sim G_{n,p}$ with high probability the lower bound in Theorem~\ref{thm:Ramsey.degree.bounds} is tight 
for all values of $q$ if $\frac{\log n}{n} \ll p \ll n^{-2/3}$,
and for some but not all values of $q$ if $n^{-2/3} \ll p \ll n^{-1/2}$. 
For complete graphs $H=K_t$ and $q\geq 3$ colours, different bounds on
$s_q(K_t)$ are proven in~\cite{attwa2025improved,bamberg2022minimum,fox2016minimum,guo2020packing,han2018vertex}. 
Further results on $s_q(H)$ can be found in e.g.~\cite{boyadzhiyska2022thesis,fox2014ramsey,grinshpun2017minimum,grinshpun2015thesis,gupta2023thesis}.

\subsection{Game versions of Ramsey results}
From now on, we consider the $(1:b)$ \emph{Maker-Breaker $H$-game} on a graph $G$.
In this game, the board $X$ is the edge set of a given graph $G$,
and Maker wins if she claims all edges of a copy of the given graph $H$.
Beck~\cite{beck2002positional,beck2008combinatorial}, Gebauer~\cite{gebauer2013size}, as well as Ne{\v{s}}et{\v{r}}il and Valla~\cite{nevsetvril2010ramsey} already considered this game
as a natural game variant of the just described Ramsey problem of asking whether $G\rightarrow_q H$ holds.

Following the Ramsey notation above, 
let us write $G \MBarrow_b H$ if Maker wins the $(1:b)$ biased $H$-game on $G$ as first player,
and in this case say that $G$ is an \emph{MB-winnable} graph for $H$ with bias $b$.
Then one of the central and impressive results in~\cite{beck2008combinatorial,beck2019two} 
states that $K_n \MBarrow_1 K_k$ holds if and only if
$k \leq \lfloor 2 \log_2 n - 2 \log_2 \log_2 n+2\log_2 e - 3+o(1) \rfloor$.
Moreover, the main result by Bednarska and {\L}uczak~\cite{bednarska2000biased} states that
for arbitrary graphs $H$ with at least 2 edges, we have $K_n \MBarrow_b H$ if $b\leq c_1 n^{1/m_2(H)}$, and
$K_n \MBnotarrow_b H$ if $b\geq c_2 n^{1/m_2(H)}$, where $m_2(H)=\max\left\{\frac{e(F)-1}{v(F)-2}:~F\subseteq H,~v(F)\geq 3 \right\}$ is the maximum 2-density of $H$,
and $c_1,c_2$ are constants depending on $H$ only.
For further results on $H$-games see also~\cite{balogh2011chvatal,gebauer2012clique,glazik2022new,muller2014threshold,nenadov2016threshold,sowa2025constructive,sowa2026constructive}.
Moreover, note that by a strategy stealing argument, the following connection between Ramsey problems and Maker-Breaker games can be proven easily; see Section~\ref{sec:concluding}.

\begin{prop}\label{prop:stealing}
Let $G,H$ be  graphs. Then $G\rightarrow_2 H$ implies $G\MBarrow_1 H$.
\end{prop}

Now, motivated by the research of Burr, Erd\H{o}s and Lov\'asz, and of subsequent papers,
we aim to analyze graphs $G$ which are minimal for winning an $H$-game.
Formally, we say that $G$ is \emph{minimal MB-winnable} for $H$ with bias $b$,
if $G \MBarrow_b H$ and  $G' \MBnotarrow_b  ~ H$ for every proper subgraph $G'$ of $G$.
We let $\mathcal{M}_{MB}^b(H)$ denote the set of all such graphs $G$,
and we define a minimal minimum degree by
$$
s_{MB}^b(H) := \min\{\delta(G):~ G\in \mathcal{M}_{MB}^b(H)\} .
$$
Also, we may be interested in analyzing games in which Breaker is the first player. 
For this, we replace MB by BM in all the definitions above.

\medskip

Our first result then is a game analogue of Theorem~\ref{thm:Ramsey.finite}.
Interestingly, we can observe that the question whether infinitely many winnable graphs exist,
does not depend on Maker being first or second player. Additionally, we see that
the characterization in (III) is different to the characterization in the Ramsey setting, see Theorem~\ref{thm:Ramsey.finite}, and neither of these characterizations describes a
subset of the other, although Proposition~\ref{prop:stealing} holds.

\begin{thm}{\label{finiteVsInfinite}}
    For a given graph $H$, the following three statements are equivalent.
    \begin{itemize}
        \item[(I)]{$\mathcal{M}^1_{MB}(H)$ consists of infinitely many graphs.}
        \item[(II)]{$\mathcal{M}^1_{BM}(H)$ consists of infinitely many graphs.}
        \item[(III)]{At least one connected component of $H$ contains at least three edges.}
    \end{itemize}
\end{thm}

We then turn to the study of graph parameters for graphs in $\mathcal{M}_{MB}^b(H)$
and $\mathcal{M}_{BM}^b(H)$, and observe that we can prove
a game analogue of Theorem~\ref{thm:Ramsey:clique.number},
and we can also prove an analogous statement for chromatic numbers.

\begin{thm} \label{thm:construction.chromatic.clique}
For every graph $H$ and every bias $b\in\mathbb{N}$,
there exists a graph $G\in \mathcal{M}_{MB}^b(H)$ such that
$\omega(G)=\omega(H)$ and $\chi(G)=\chi(H)$.
\end{thm}

In fact, the case $b=1$ in the above theorem already follows from Theorem 1.4 in~\cite{nevsetvril2010ramsey},
and, if we would consider the clique number only, it could be proved by using 
Theorem~\ref{thm:Ramsey:clique.number} and Proposition~\ref{prop:stealing}.
Moreover, we note that the result above turns out to be useful for our discussion on minimum degrees,
as given by the following theorem.

\begin{thm}\label{thm:general.bounds}
Let $H$ be a connected graph and $b\in\mathbb{N}$. Then the following holds:
\begin{enumerate}
\item[(i)] $2\delta(H) - 1 \leq s_{BM}^1(H) \leq s_{MB}^1(H) \leq 2\Delta (H) - 1$,
\item[(ii)] $(b+1)(\delta(H) - 1) + 1 \leq s^b_{BM}(H) \leq s^b_{MB}(H) \leq  b\Delta(H) (\ln(\Delta(H)) + 2)$.
\end{enumerate}
\end{thm}

Note that the lower bounds in Theorem~\ref{thm:general.bounds} mirror the lower bound
of Theorem~\ref{thm:Ramsey.degree.bounds}, when we set $q=b+1$, i.e.~when the total number of Maker's edges at the end of the game is equal to the average size of a colour class in an edge-colouring with $q$ colours.
Moreover, for most graphs $H$, the upper bound in Theorem~\ref{thm:general.bounds} is much smaller
than the upper bound in Theorem~\ref{thm:Ramsey.degree.bounds}. 
If $b=1$, then for all regular graphs $H$ we obtain precise results, e.g.~$s_{MB}^1(K_t)=2t-3$ and $s_{MB}^1(C_t)=3$.

Next to this, we find the following precise results.

\begin{thm}\label{thm:some.precise.degrees}
The following holds:
\begin{enumerate}
\item[(a)] For complete bipartite graphs, we have $s_{MB}^1(K_{s,t}) = s_{BM}^1(K_{s,t}) = 2\min\{s,t\} - 1 .$
\item[(b)] For the wheel $W_t$ with $t$ spokes, we have $s_{MB}^1(W_t) = s_{BM}^1(W_t) = 5 .$
\end{enumerate}
\end{thm}

\begin{thm}\label{thm:some.precise.degrees.biased}
Let $b\geq 1$, then the following holds:
\begin{enumerate}
\item[(a)] For every cycle $C_t$, we have $s_{MB}^b(C_t) = s_{BM}^b(C_t) = b+2 .$
\item[(b)] For the wheel $W_t$ with an even number $t$ of spokes, we have $s_{MB}^b(W_t) = s_{BM}^b(W_t) = 2b+3 .$
\item[(c)] For every tree $T$, we have $s_{MB}^b(T) = s_{BM}^b(T) = 1 .$
\end{enumerate}
\end{thm}

\medskip

\textbf{Organization of the paper.} In Section~\ref{sec:prelim} we define our notation and collect two winning criteria for Maker-Breaker games that are useful for the analysis of later strategies. 
In Section~\ref{sec:existence} we prove Theorem~\ref{finiteVsInfinite};
then Theorem~\ref{thm:construction.chromatic.clique} is proved in
Section~\ref{sec:clique.chromatic},
and the proofs of Theorems~\ref{thm:general.bounds},~\ref{thm:some.precise.degrees} and~\ref{thm:some.precise.degrees.biased} follow in Section~\ref{sec:degrees}.
We end the paper with some concluding remarks and open problems in Section~\ref{sec:concluding}.

\medskip

\section{Preliminaries and Notation} \label{sec:prelim}

We use standard graph theoretic notation. Given a graph $G$, we denote the set of its vertices with $V(G)$ and the set of its edges with $E(G)$. Let $v(G)$ be the size of $V(G)$ and similarly $e(G)$ the size of $E(G)$. Moreover, we denote the degree of a vertex $v$ of $G$ with $d_G(v)$, the minimum degree among all vertices of $G$ as $\delta(G)$ and the maximal degree as $\Delta(G)$. We denote the size of the largest clique of $G$ with $\omega(G)$. A colouring of the vertices of $G$ is caleld proper, if no two adjacent vertices share a colour. The \emph{chromatic number} $\chi(G)$ of $G$ is the smallest number of colours of a proper colouring of $G$.  We also define the maximum $2$-density $m_2(H)=\max\left\{\frac{e(F)-1}{v(F)-2}:~F\subseteq H,~v(F)\geq 3 \right\}$.

Let $n \in \mathbb{N}$, we define $[n]\coloneqq \{1, \dots, n\}$. We denote the complete graph on $n$ vertices as $K_n$, the path on $n$ vertices $P_n$ and the cycle on $n$ vertices as $C_n$. We denote a star with a central vertex and $n$ leaves as $S_n$ and a wheel with a central vertex and $n$ vertices, forming a cycle, adjacent to it as $W_n$. For $r,s \in  \mathbb{N}$ the graph $K_{r,s}$ is the complete bipartite graph with one bipartition class of size $r$ and one of size $s$. 

Given two graphs $F$ and $G$, we write $H = F + G$ if $H$ consist of vertex disjoint copies of $F$ and $G$ and there are no edges in between those copies. For a natural number $k$ we write $kF = F + F + \ldots + F$ for the graph consisting of $k$ vertex-disjoint copies of $F$. If $f$ is an edge of $F$, we denote with $F - f$ the graph on the vertex set $V(F)$ and with edge set $E(F) \setminus \{f\}$. If $v$ is a vertex in $V(F)$, we denote the graph on the vertex set $V(F) \setminus \{v\}$ and edge set $E(F) \setminus \{e \in E(F): v \in e\}$ as $F - v$.

Given a hypergraph $\mathcal{H}$, let $V(\mathcal{H})$ be its vertex-set and $E(\mathcal{H})$ the set of its hyperedges.
We set $v(\mathcal{H})=|V(\mathcal{H})|$ and $e(\mathcal{H})=|E(\mathcal{H})|$. If every hyperedge in $E(\mathcal{H})$ consists of $d$ vertices for some $d \in \mathbb{N}$, we say that $\mathcal{H}$ is $d$-\textit{uniform}. 
When $V_0$ is a subset of $V(\mathcal{H})$, we write $e_\mathcal{H}(V_0)$ for the number of all hyperedges in $\mathcal{H}$ containing only vertices in $V_0$.
We write $\Delta_2(\mathcal{H})$
for the maximum pair degree in $\mathcal{H}$,
i.e.~the largest number of hyperedges in $\mathcal{H}$ that two fixed vertices belong to.
Additionally, if two hypergraphs $\mathcal{H}_1$ and $\mathcal{H}_2$ are isomorphic, we write
$\mathcal{H}_1 \cong \mathcal{H}_2 $.

Finally, consider a Maker-Breaker game on the edge-set of a graph $G$. During the game we say that an edge is \emph{free} if it is claimed by no player yet. Moreover, we say that a player applies a \emph{pairing strategy} with a family of 
pairwise disjoint pairs $\mathcal{P}_1,\ldots,\mathcal{P}_k \subseteq E(G)$, if this player claims an element of a pair $\mathcal{P}_i$ whenever the opponent claimed an edge of the same pair in his/her previous move.

\medskip

We make use of the game
$\text{Box}(p,1;a_1,\ldots,a_n)$,
as it was introduced by Chv\'atal and Erd\H{o}s~\cite{chvatal1978biased}
but following the notation in~\cite{hefetz2014positional}.
The game is played on a hypergraph $(X,\mathcal{H})$ with $\mathcal{H}=\left\{F_1,\ldots,F_n\right\}$
such that the winning sets (boxes) are
pairwise disjoint and satisfy $|F_i|=a_i$ for every $i\in [n]$.
In every round, Maker claim $p$ elements from $X$, while Breaker claims only 1. As usual, Maker wins
if and only if she claims all elements of a winning set.

The following lemma gives a winning criterion for Breaker when all boxes have the same size.

\begin{thm}[Criterion for Box Game,~Theorem 2.1 in~\cite{chvatal1978biased}] \label{thm:Box.game}
If $a_i = m$ for every $1\leq i \leq n$, and $m > p \sum_{i=1}^n 1/i$,
then Breaker has a winning strategy for $\text{Box}(p,1;a_1,\ldots,a_n)$.
\end{thm}

In fact, we may use this statement in the following situation.
Assume that in a game on a graph $G$,
we find a vertex $w\in V(G)$ and 
disjoint vertex sets $A_j$, with $j\in [\Delta-1]$ for some integer $\Delta$,
such that each set has size $|A_j|=s$ and there are free edges between $w$ and all vertices of the $A_j$. Then, by the above criterion it follows that, 
in a $(1:b)$ Maker-Breaker game, Maker (imagining to be Breaker in a Box Game) can claim one edge between
$w$ and each of the sets $A_j$ if $s \geq  b \ln(\Delta) + 1$.

\smallskip

Next to this, we make also use of the following general winning criterion for Maker. 

\begin{thm}[Weak Win Criterion,~Theorem 2 in~\cite{beck1982remarks}] \label{thm:Weak.Win}
Let a hypergraph $(X,\mathcal{F})$ and a bias $b\in\mathbb{N}$ be given. If
$$
\sum_{F\in \mathcal{F}} (1+b)^{-|F|} > \frac{b^2}{(1+b)^3} \Delta_2(\mathcal{F}) |X|
$$
holds, Maker (as first player) has a winning strategy for the $(1:b)$ Maker-Breaker game on 
$(X,\mathcal{F})$.
\end{thm}

\medskip

\section{Existence of minimal winnable graphs} \label{sec:existence}

In this section, we prove Theorem~\ref{finiteVsInfinite}. 
As we only consider unbiased games here, we drop the bias $b$ in all notations throughout the chapter, e.g.~we write $\mathcal{M}_{MB}(H)$ instead of $\mathcal{M}^1_{MB}(H)$,
and $G \MBarrow H$ instead of $G \MBarrow_1 H$.

\medskip

\subsection{Forests with small components} \label{sec:finite.case}

In this subsection we consider graphs $H$ with at most two edges in every component and prove the following Lemma.

\begin{lemma}\label{lem:finite:matching.and.cherries}
    Let $H := a P_2 + b P_3$ with $a, b \in \mathbb{N}_0$, then 
	$\mathcal{M}_{MB}(H)$ and $\mathcal{M}_{BM}(H)$ ar both finite.
\end{lemma}

\begin{proof}

The claim is trivial for $a=b=0$, so we may assume that  $a>0$ or $b>0$.
For notational convenience set $k=6b+4a$ and $t=2b+2a+1$.
Let $G$ be an arbitrary graph in $\mathcal{M}_{MB}(H)$ or in $\mathcal{M}_{BM}(H)$.
To show that those sets are finite we aim to prove that $v(G)\leq(2t+a)\cdot(24k^3t+4)$.
In order to prove this upper bound on $v(G)$, we first prove a sequence of simple claims.

\begin{clm} \label{clm:finite:matching.and.cherries:max.degree}
$\Delta(G) \leq k$
\end{clm}

\begin{proof}
For a contradiction, suppose that there is some $v\in V(G)$ with $d_G(v)\geq k+1$.

Set $H'\coloneqq (a-1)P_2+bP_3$ if $b=0$ and $H'\coloneqq aP_2+(b-1)P_3$ if $b\geq 1$.
Note that in both cases we have $H\subseteq H'+P_3$ and $v(H')\leq  \frac{k}{2}-2$.
Moreover if a graph $F$ contains a copy of $H$, then for any $w\in V(F)$ the graph $F-w$ must contain a copy of $H'$. 

Now consider the unbiased game on $G-v$ with Maker moving first if $G \in \mathcal{M}_{MB}(H)$ and Breaker moving first otherwise. By the last property, Maker must have a strategy to claim a copy of $H'$ on $G-v$, since otherwise she cannot claim a copy of $H$ in the game on $G$.

Next, let $e\in E(G)$ be any edge incident to $v$.
We prove that Maker has a strategy to claim $H$ in the game on $G-e$:
Maker divides the edges of $G-e$ into two disjoint boards. The first consists of the $d_G(v)-1 \geq k$ edges incident to $v$ in $G-e$, the other of all the edges of $G-v$.
When Maker starts the game, she starts with an edge of $G-v$, afterwards Maker always responds to Breaker's move by playing on the same board. Doing so, she plays her winning strategy for $H'$ on $G-v$ and a pairing strategy on the edges incident to $v$.
The latter guarantees that she obtains at least $\lfloor\frac{d_G(v)-1}{2}\rfloor \geq \frac{k}{2}$ of these edges.
Since $\frac{k}{2}-v(H')\geq 2$, at least two of the edges do not intersect her copy of $H'$ in $G-v$.
Thus she manages to claim a copy of $H'+P_3\supseteq H$, as needed.

Hence $G-e$ is winnable for $H$, which gives the required contradiction to the minimality of $G$.
\end{proof}

\begin{clm} \label{clm:finite:matching.and.cherries:P3}
$G$ contains no family of $t$ pairwise vertex-disjoint subgraphs isomorphic to 
$P_4$, $K_3$ or $K_{1,3}$.
\end{clm}

\begin{proof}
Assume $G$ has $t$ pairwise vertex-disjoint subgraphs $G_1,\ldots,G_t$ 
that are isomorphic to $P_4$, $K_3$ or $K_{1,3}$. 
Then consider the following strategy for Maker:
Fix the pairs $\{G_{2i-1},G_{2i}\}$ with $i\in [a+b]$
and play as follows. 
If Maker starts the game, she skips her first turn. 
Then, whenever Breaker claims a first edge in one of the 
mentioned pairs,
w.l.o.g.~let it be the graph $G_{2i-1}$ for some $i\in [a+b]$, then Maker 
claims an edge in the other graph $G_{2i}$,
with the restriction that it is the middle edge if $G_{2i}$ is 
isomorphic to $P_4$. Moreover, when Breaker later claims another edge 
of the pair $\{G_{2i-1},G_{2i}\}$, Maker claims a second edge in $G_{2i}$,
thus completing a copy of $P_3$.

This way, Maker claims $a+b$ vertex-disjoint copies of $P_3$,
and hence wins the $H$-game.
However, as we did not use $G_t$ for the strategy,
we see that $G$ is not minimal BM-winnable and not minimal MB-winnable for $H$,
a contradiction.
\end{proof}

\begin{clm} \label{clm:finite:matching.and.cherries:components}
Every component of $G$ has at most $24k^3t+4$ vertices.
\end{clm}

\begin{proof}
As $\Delta(G) \leq k$, we have that every edge of $G$ belongs to at most $3k^2$
copies of $P_4$. When we delete all the vertices of one copy of $P_4$ in $G$ 
(and hence delete at most $4k$ edges of $G$ by the degree bound),
we can delete at most $12k^3$ copies of $P_4$. 
Thus, using Claim~\ref{clm:finite:matching.and.cherries:P3} it follows that $G$
cannot have more than $12k^3t$ copies of $P_4$ in $G$.

Now, consider any component $C$ of $G$. If $C$ is a star, then $v(C)\leq k+1$ by 
Claim~\ref{clm:finite:matching.and.cherries:max.degree}.
Otherwise $C$ contains a copy of $P_4$ or a copy of $K_3$, fix one such subgraph and call it $F$.
For every vertex $v$ in $V(C)\setminus V(F)$, there is a path from $v$ to $F$, as $C$ is connected, and 
this path can be extended by at least 2 edges of $F$. That is, $v$ is the endpoint of a copy of $P_4$. 
Since the number of $P_4$-copies is bounded by $12k^3t$,
there can be at most $24k^3t$ endpoints of $P_4$-copies,
and hence the claimed bound on $v(C)$ follows.
\end{proof}

To finish the argument, let $\bar{C}_1,\ldots,\bar{C}_r$ be the components of $G$ which are not stars, 
and let $C_1^\ast,\ldots,C_s^\ast$ be the components of $G$ which are stars.
Each of the components $\bar{C}_1,\ldots,\bar{C}_r$ contains a copy of $P_4$ or $K_3$, hence $r\leq t$ by Claim~\ref{clm:finite:matching.and.cherries:P3}.
Among the stars $C_1^\ast,\ldots,C_s^\ast$ there can be at most $t$ stars with more than 2 edges, again by Claim~\ref{clm:finite:matching.and.cherries:P3}. 
Moreover, since $P_3 \not\MBarrow P_3$, the remaining star components can only help to create the $a$ matching edges of $aP_2+bP_3$, and by minimality of $G$, their number cannot be larger than $2a$,
as $2a$ matching edges are enough for claiming a matching of size $a$.
As every component has at most $24k^3t+4$ vertices, we conclude
$v(G) \leq (2t + 2a)\cdot (24k^3t+4)$.
\end{proof}

\medskip

\subsection{Infinitely many minimal winnable graphs} \label{sec:infinite.case}

The goal of this subsection is to show that $\mathcal{M}_{MB}(H)$ and $\mathcal{M}_{BM}(H)$ are both infinite if $H$ has a component with at least three edges. 
Before we look at concrete graphs, we proof the following more general result.

\begin{lemma}\label{lem:MBBM-Lemma}
Let $H$ be a connected graph. If $\mathcal{M}_{MB}(H)$ contains infinitely many graphs, then $\mathcal{M}_{BM}(H)$ does as well. 
\end{lemma}
   
\begin{proof}
    Assume $H$ is a fixed connected graph such that $\mathcal{M}_{MB}(H)$ is infinite. To show that $\mathcal{M}_{BM}(H)$ is also infinite, we aim to show that for any $n\in\mathbb{N}$, there is some graph $G\in \mathcal{M}_{BM}(H)$ such that $v(G)\geq n$.

    Since $\mathcal{M}_{MB}(H)$ is infinite, we first find $G'\in\mathcal{M}_{MB}(H)$ with $v(G')\geq \frac{n}{2}$. Now consider playing the Breaker-Maker $H$-game on $G :=2G'$. Breaker can only select an edge from one of the two copies of $G'$ in his first move. But $G'$ is MB-winnable for $H$, so Maker can force a victory by being the first to play on the untouched copy.

    Next, we show minimality of $G$. So, let $e \in E(G)$ be an edge of $G$. We consider a Breaker-Maker game on the graph $G-e$. By definition, $e$ must be in exactly one of the two copies of $G'$. Breaker starts the game by claiming an arbitrary edge of the other copy. Now, considering Breaker always answers on the same copy Maker played, he has a strategy to win the $H$-game on both copies. Since $H$ is connected that means, that Maker cannot get a copy of $H$ at all and looses the game.
\end{proof}

Next, we show that for every connected graph $H$ with $3$ or more edges $\mathcal{M}_{MB}(H)$ contains of infinitely many graphs. By Lemma~\ref{lem:MBBM-Lemma} this implies that also $\mathcal{M}_{BM}(H)$ is infinite for those graphs.

\smallskip

We start by looking at the star $S_r$.
Afterwards we consider $S_{r,s}$, defined as the double star consisting of a star with $r$ leaves and a star with $s$ leaves plus an extra edge between the centres of those two stars.
 Finally, we deal with $S_{r,s,t}$,  the triple-star consisting of a star of size $r$, a star of size $s$ and a star of size $t$, plus two edges between the centre of the second star and the centres of the other two stars.
 Then we use an inductive argument, using the three previous graphs for the base case, to prove the claim for all trees (see Lemma \ref{lem:trees}). Afterwards we consider all connected graphs that contain a cycle (see Lemma \ref{cycles}).

\begin{lemma}\label{lem:star}
    Let $r \geq 3$. Then $\mathcal{M}_{MB}(S_r)$ consists of infinitely many graphs.
\end{lemma}

\begin{proof}
Let $\ell \geq 2$ be a natural number. We construct a graph $G_\ell$ as follows; see also Figure~\ref{fig:star.construction}. We start with a path $P_\ell$ of length $\ell - 1$, and vertices $x_1, x_2, \ldots, x_{\ell}$. Then, we attach a star $S_{x_1}$ of size $2r - 3$ with centre $x_1$, a star $S_{x_{\ell}}$ of size $2r - 3$ with centre $x_{\ell}$, and a star $S_{x_i}$ of size $2r - 5$ with centre $x_i$ for every $2\leq i \leq \ell - 1$. We claim that for every $\ell$, Maker wins the $S_r$-game on $G_\ell$, but loses it on $G_\ell - e$ for every edge $e$ in the path $P_\ell$. 

\begin{center}
\begin{figure}[t] 
	\begin{center}
\includegraphics[page=1,scale=0.6]{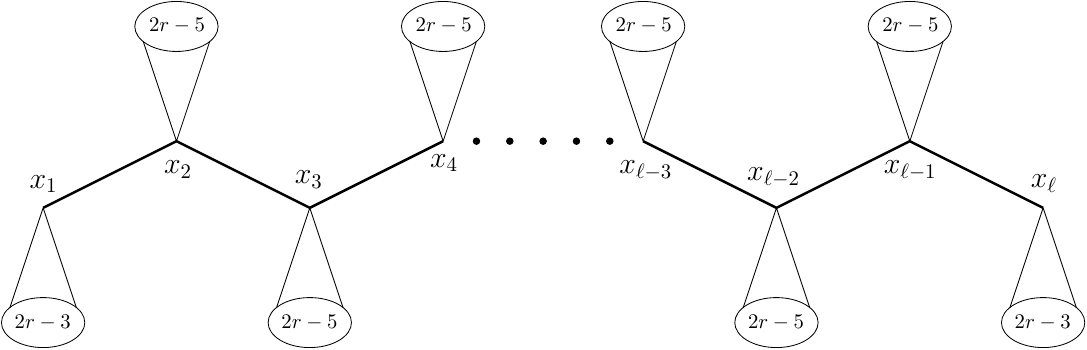}
	\end{center}
	\caption{Construction for the $S_r$-game.}
	\label{fig:star.construction}
\end{figure}
\end{center}

\begin{clm} \label{clm:stars.claim1}
$G_\ell \MBarrow S_{r}$. 
\end{clm}

\begin{proof}
Maker starts her strategy by claiming the edge $x_1x_2$. 
If Breaker responds with an edge that is not part of the star $S_{x_1}$, Maker afterwards claims an edge of $S_{x_1}$ herself. Then there are $2r - 4$ free edges left in $S_{x_1}$. By a pairing strategy, Maker can get $r - 2$ of them, and finish a copy of $S_r$.

Hence, from now on we may assume that Breaker claims an edge in $S_{x_1}$ in his first turn. 
Then in the second round, Maker claims the edge $x_2x_3$. Then, as longs as Breaker claims an edge of $S_{x_i}$ in round $i$ for $i\in [\ell - 2]$, Maker responds by claiming the edge $x_{i + 1}x_{i + 2}$ in round $i + 1$. Once there is a round $i$, where Breaker does not play in the star $S_{x_i}$, Maker plays in the star $S_{x_i}$ in round $i + 1$. Afterwards, there are $2r - 6$ free edges left in $S_{x_i}$. Again by pairing, Maker can get $r - 3$ of these edges, finishing a copy of $S_r$ together with the edges 
$x_{i - 1}x_{i}$, $x_i x_{i + 1}$ and the edge that Maker claimed in round $i + 1$.

\smallskip

Hence, we may assume from now on that Breaker
claims an edge of $S_{x_i}$
in round $i$ for every $i\in [\ell-2]$.
Then, in round $\ell - 1$, Maker claims the edge $x_{\ell - 1} x_\ell$ and creates a double-threat. If Breaker does not respond in $S_{x_{\ell - 1}}$, Maker can claim an edge there and wins by pairing on $S_{x_{\ell - 1}}$, as before. Otherwise, Breaker does not play in $S_{x_{\ell}}$ and Maker wins analogously by claiming edges of $S_{x_{\ell}}$.
\end{proof}

\begin{clm} \label{clm:stars.claim2}
Let $e$ be an edge in $P_\ell$. Then $G_\ell - e\MBarrow S_{r}$ does not hold.    
\end{clm}

\begin{proof}
Let $e = x_ix_{i + 1}$ with $i \in [\ell-1]$. Then, the graph $G_\ell - e$ consists of two connected components. It is sufficient to show that Breaker has a strategy to win the $S_r$-game on each of them separately, and due to symmetry it is sufficient to only look at the component containing $x_1$.
Breaker can win by a pairing argument on this component. He builds $r - 2$ pairs in $S_{x_1}$ and pairs the remaining free edge of $S_{x_1}$ with the edge $x_1 x_2$. For every $2\leq j\leq i - 1$, Breaker builds $r - 3$ pairs in $S_{x_j}$ and pairs the remaining edge of $S_{x_j}$ with $x_jx_{j+1}$. Finally, he builds $r - 3$ pairs in $S_{x_i}$, leaving the last free edge unpaired. By following that strategy, Breaker ensures that Maker cannot get a vertex of degree $r$, and hence no copy of $S_r$.
\end{proof}

The previous two claims imply that for every $\ell\geq 2$, the graph $G_\ell$ contains a minimal MB-winnable graph for $S_r$ that uses all edge of $P_\ell$ and hence has at least $\ell$ vertices. As $\ell$ is arbitrary, this leads to infinitely many minimal MB-winnable graphs for $S_r$, finishing the proof of Lemma~\ref{lem:star}.
\end{proof}

\medskip

\begin{lemma}\label{lem:double-star}
    Let $r, s \geq 1$ be positive integers. Then $\mathcal{M}_{MB}(S_{r,s})$ consists of infinitely many graphs.
\end{lemma}

\begin{proof}
Let $\ell \in \mathbb{N}$ be an even number. We again start the construction of a graph $G_\ell$ with a path $P_\ell$ on $\ell$ vertices $x_1, \ldots , x_\ell$. Then, we attach a star $S_{x_1}$ with $2r$ leaves at $x_1$ and a star $S_{x_\ell}$ with $2s$ leaves at $x_\ell$. Moreover, for all $2\leq i\leq \ell - 1$, we attach a star $S_{x_i}$ at $x_i$, which has $2s-2$ leaves if $i$ is even, and $2r - 2$ leaves if $i$ is odd. See also Figure~\ref{fig:double.star.construction} for this construction. Similarly to the previous proof, we intend to show that Maker wins the $S_{r,s}$-game on $G_\ell$, but loses on $G_\ell - e$ for every edge $e$ in the path $P_\ell$. 

\begin{center}
\begin{figure}[t] 
	\begin{center}
\includegraphics[page=2,scale=0.6]{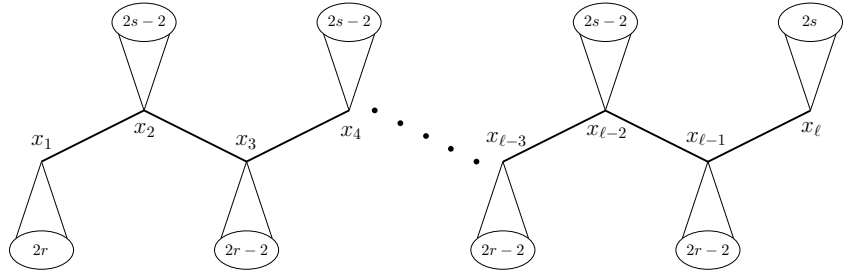}
	\end{center}
	\caption{Construction for the $S_{r,s}$-game.}
	\label{fig:double.star.construction}
\end{figure}
\end{center}

\begin{clm}
$G_\ell \MBarrow S_{r,s}$. 
\end{clm}

\begin{proof}
Outside of the path $P_\ell$ Maker can apply a pairing strategy to ensure to get half of the edges in every star $S_{x_i}$. Additionally we claim that, playing on $P_\ell$, Maker can get a path of length $3$, or a path of length $2$ starting in $x_1$ or in $x_\ell$. 

Assume first that Maker claims a path from $x_{j - 1}$ to $x_{j + 2}$ for some $3\leq j \leq \ell - 3$. Then, the edges of that path, together with Maker's edges of the stars $S_{x_j}$ and $S_{x_{j + 1}}$ form a copy of $S_{r,s}$. 
If otherwise Maker claims the path from $x_1$ to $x_3$, she creates a copy of $S_{r,s}$, together with her edges in the stars $S_{x_1}$ and $S_{x_2}$. Analogously, Maker gets a copy of $S_{r,s}$, if she claims the path from $x_{\ell-2}$ to $x_\ell$ instead.

Hence, it remains to show that playing on $P_\ell$ only, Maker can claim a subpath as described.
To do so, Maker starts her strategy on $P_\ell$ by claiming the edge $x_2 x_3$.
If the next time Breaker plays on the path he does not claim the edge $x_1 x_2$, Maker claims it and wins. Hence, Breaker next move on the path must be $x_1 x_2$. In this case Maker responds with the edge $x_4 x_5$ as her second edge of $P_\ell$, and analogously forces Breaker to claim $x_3 x_4$. In general, for every $i \in [\frac{\ell}{2} - 1]$ Maker can claim the edge $x_{2i}x_{2i+1}$ as her $i$-th edge of $P_\ell$ and force Breaker to respond with the edge $x_{2i - 1} x_{2i}$. Doing so, Maker eventually claims the edge $x_{\ell - 2} x_{\ell - 1}$ as her $\frac{\ell}{2} - 1$-th edge of $P_\ell$, and creates a double-threat: Maker can now claim a path as described by claiming either the edge $x_{\ell - 3}x_{\ell - 2}$ or the edge $x_{\ell - 1}x_{\ell}$. Breaker cannot block both edges with his next edge of $P_\ell$ and hence Maker wins.
\end{proof}

\begin{clm}
Let $e$ be an edge in $P_\ell$. Then $G_\ell - e\MBarrow S_{r,s}$ does not hold.
\end{clm}

\begin{proof}
Let $e = x_i x_{i + 1}$ with $i \in [\ell - 1]$. As in the proof of 
Claim~\ref{clm:stars.claim2},
it is enough to show that Breaker wins the $S_{r,s}$-game on the connected component of $G_\ell - e$ that contains $x_1$. To do so, Breaker applies the following pairing: He pairs all edges inside the added stars, getting half of the edges of each of the stars $S_1,\ldots,S_{x_i}$. Additionally, he pairs the edges $x_{2j-1}x_{2j}$ and $x_{2j} x_{2j+1}$ for every $j \in [\lceil\frac{i}{2}\rceil-1]$, 
and thus ensures that Maker cannot get a subpath of $P_\ell$ of length $3$, or a subpath of length $2$ starting in $x_1$. 
Together with Maker's stars attached to these Maker's subpaths of $P_\ell$,
the only double stars
that she can get are isomorphic to subgraphs of
$S_{r,s-1}$ or $S_{r-1,s}$,
and hence Maker loses the $S_{r,s}$-game.
\end{proof}

As in Lemma~\ref{lem:star}, the statement of Lemma~\ref{lem:double-star} now follows from the previous two claims.
\end{proof}

\medskip

\begin{lemma}\label{lem:triple-star}
    Let $r, s, t \in \mathbb{N}_0$ such that $r, t \geq 1$. Then $\mathcal{M}_{MB}(S_{r,s,t})$ consists of infinitely many graphs.
\end{lemma}

\begin{proof}
Assume w.l.o.g.~that $r \geq t$ holds. Let 
$m = \max\{r,s,t\}$ and let $\ell \in \mathbb{N}$ such that $\ell \equiv 2$ ($\text{mod }4$) holds. We start the construction of $G_\ell$ with a path $P_\ell$ of length $\ell - 1$ and vertices $x_1,\ldots,x_\ell$. Then we add two triangles $\{x_1, x_0, x_{-1}\}$ and $\{x_\ell, x_{\ell + 1}, x_{\ell + 2}\}$ containing the two endpoints of our path. Finally we attach stars with centres $x_i$, for every $-1\leq i \leq \ell + 2$, such that $S_{x_i}$ has size $2m$ if $i \equiv 1$ or $i\equiv 2$ ($\text{mod }4$) holds, and size $2t - 2$ otherwise. See also Figure~\ref{fig:triple.star.construction}.
In the following, for short, we call $x_i$ a big vertex if $S_{x_i}$ has $2m$ edges, and small otherwise.
As in previous proofs, we intend to show 
that $G_\ell \MBarrow S_{r,s, t}$ holds and $G_\ell - e \MBarrow S_{r,s, t}$ does not hold for any edge $e$ in the path $P_\ell$.

\begin{center}
\begin{figure}[t] 
	\begin{center}
\includegraphics[page=3,scale=0.6]{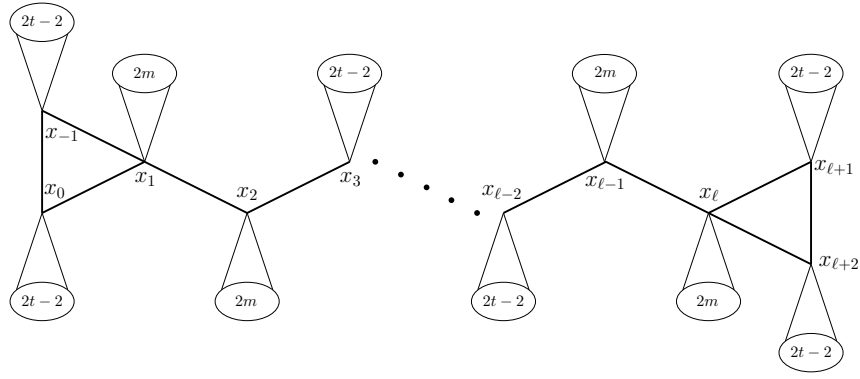}
	\end{center}
	\caption{Construction for the $S_{r,s,t}$-game.}
	\label{fig:triple.star.construction}
\end{figure}
\end{center}

\begin{clm}
$G_\ell \MBarrow S_{r,s,t}$.     
\end{clm}

\begin{proof}
Maker applies a pairing strategy in all the attached stars, meaning that she gets $m$ edges in the stars of size $2m$ and $t-1$ edges in the stars of size $2t - 2$. We call a path of length $3$ in the vertex set $\{x_i: -1\leq i \leq \ell+2\}$ useful, if in the path two adjacent vertices are big and two adjacent vertices are small. Assume that Maker can claim a useful path, say with vertices $x_i,x_{i+1},x_{i+2},x_{i+3}$
such that w.l.o.g.~$x_i$ is big
and $x_{i+3}$ is small,
then the edges of that path together with Maker's stars at 
$x_i,x_{i+1},x_{i+2}$ form a copy of $S_{m,m,t}$ which contains a copy of $S_{r,s,t}$. Hence, it remains to show that Maker can claim such a useful path in a game played only on the graph induced by $\{x_i: -1\leq i \leq \ell+2\}$.

In order to do so, Maker pairs the edge $x_{-1}x_1$ with $x_0x_1$, and
the edge $x_\ell x_{\ell+1}$ with $x_\ell x_{\ell+2}$.
Apart from those pairings, Maker plays the following strategy.
On her first turn she claims the edge $x_1 x_2$. Suppose Breakers first move that is not part of any of the described pairings is not $x_{-1} x_0$. Then Maker claims that edge and gets a useful path by using the edge she gets from the pair $(x_{-1}x_1,x_0x_1)$. 
Hence, Breaker is forced to claim
$x_{-1} x_0$ first (except from claiming edges in the pairings),
and we may assume from now on that this is Breaker's first edge.
Maker responds with the edge $x_3 x_4$, forcing Breaker to claim the edge $x_2 x_3$ to prevent Maker from getting a useful path with endpoints $x_1$ and $x_4$. Maker continues with this strategy, meaning that she claims the edge $x_{2i - 1}x_{2i}$ in round $i$, forcing Breaker to respond with the edge $x_{2i - 2} x_{2i - 1}$, for every $i\in [\ell/2]$.
Finally, Maker can claim
$x_{\ell + 1}x_{\ell + 2}$ and obtains a useful path
by using the edge she gets from the pair $(x_{\ell}x_{\ell +1},x_\ell x_{\ell +2})$. 
\end{proof}

\begin{clm}
Let $e$ be an edge in $P_\ell$. Then $G_\ell - e\MBarrow S_{r,s,t}$ does not hold.
\end{clm}

\begin{proof}
Again due to symmetry, 
it is enough to show that Breaker wins the $S_{r,s,t}$-game on the component of $G_\ell - e$ which contains $x_1$. 
Let $i \in \mathbb{N}_0 \cup \{-1\}$ be the maximum number, such that the vertex $x_{2i + 2}$ is in that component. Breaker wins the game by a pairing strategy. He pairs the edges inside of the stars to get half of the edges in each star. Additionally, he considers the pairs $(x_{- 1}x_1, x_0x_1)$, $(x_{- 1}x_0, x_1x_2)$ if $i \geq 0$ and 
$(x_{2j}x_{2j+1}, x_{2j + 1} x_{2j + 2})$ for all $j \in [i]$ if $i \geq 1$. Depending on $e$, there might be an unpaired edge left.

Following that strategy, Breaker makes sure, that Maker cannot get more than two adjacent edges on the path $P_\ell$, 
but Maker may get a path of length $3$ from the vertex $x_{-1}$ or $x_0$ to the vertex $x_3$. 
Together with Maker's stars attached to these paths,
the only triple stars
that she can get are isomorphic to subgraphs of
$S_{m+1,m,t-1}$ or $S_{m,t-1,t-1}$,
but none of these can contain a copy of $S_{r,s,t}$, since $r\geq t$.
\end{proof}

Analogously to Lemma~\ref{lem:star}, the statement of Lemma~\ref{lem:triple-star} now follows from the previous two claims.
\end{proof}

Now, we are ready to prove the existence of infinitely many minimal winnable graphs for all remaining trees.

\begin{lemma}\label{lem:trees}
    Let $T$ be a tree with at least three edges. Then $\mathcal{M}_{MB}(T)$ consists of infinitely many graphs.
\end{lemma}

\begin{proof}
Throughout this proof, for any graph $F$ let $F'$ be the graph obtained by removing every leaf of $F$. In the same way $F'', F''', \dots$ are obtained by removing the leaves of the preceding graph.
We prove the statement by induction on the number of these steps needed to reduce the tree $T$ to a tree on at most two edges.
Note that Lemmas \ref{lem:star}, \ref{lem:double-star} and \ref{lem:triple-star} cover all trees $T$ for which $e(T')\leq 3$, i.e.~applying only one steps of removing all leaves is sufficient. 

Hence we may assume that both $T$ and $T'$ have at least $3$ edges and,
by induction hypothesis, that $\mathcal{M}_{MB}(T')$ is infinite. To show that $\mathcal{M}_{MB}(T)$ is infinite as well, we intend to prove that for any $t\in\mathbb{N}$ there is $H\in \mathcal{M}_{MB}(T)$ with $v(H)\geq t$.
By assumption there is a graph $G$ that is MB-minimal for $T'$ and satisfies $v(G)\geq t$.
Now let $K$ be the graph obtained by attaching $2m$ leaves to every vertex of $G$, where $m$ is the highest number of leaves adjacent to a vertex in $T$. In the remainder of the proof we refer to the edges added by this step as \emph{leaf edges}.

Maker can win the $T$-game on $K$ as follows. She considers the edges of $G$ and the set of leaf edges as sub-boards. Maker plays her first move on $G$, then she replies to each move of Breaker on the same board he played on in his previous move.
On $G$ she plays according to a winning strategy for the $T'$-game. On the leaf edges she plays a pairing strategy, ensuring she gets at least $m$ of the $2m$ leaf edges on every vertex of $G$. At the end of the game Maker's graph contains a copy of the tree obtained by attaching $m$ leaves at every vertex of $T'$. By our choice of $m$ that tree contains $T$.

Next, consider all graphs $H$ such that $G\subseteq H\subseteq K$ and Maker wins the $T$-game on $H$. Among those select an inclusion minimal $H$, i.e.~a graph $H$ for which either we have $H=G$ or deleting any leaf edge from $H$ results in a graph on which Breaker wins the $T$-game.

We now show that $H\in\mathcal{M}_{MB}(T)$. By construction, Maker wins the $T$-game on $H$ but loses it if any more leaf edges are removed.
It remains to show that deleting any edge stemming from $G$ also leads to a graph on which Breaker can win the $T$-game.
Let $e$ be such an edge and consider the $T$-game on $H-e$.
Since $G\in\mathcal{M}_{MB}(T')$,
Breaker can play in such a way that
Maker does not create a copy of $T'$ in $G-e$. But then, let $F$ be Maker's graph at the end of the game and consider $F'$.
Clearly $F'$ cannot contain any of the leaf edges, so $F'\subseteq G-e$.
Moreover, if Maker would claim a copy of $T$, then she would have a copy of $T'$ contained in $F'\subseteq G-e$, but this is what Breaker prevents by his strategy.

Therefore we have $H\in \mathcal{M}_{MB}(T)$. Moreover, we have $t\leq v(G)\leq v(H)$, and so $|\mathcal{M}_{MB}(T)|$ must be infinite.
\end{proof}

Now, we extend our results to all connected graphs that contain a cycle.

\begin{lemma}\label{cycles}
    Let $H$ be a connected graph containing a cycle. Then $\mathcal{M}_{MB}(H)$ is infinite.
\end{lemma}

\begin{proof}
Let $t\geq 1$. The proof uses a result by Rödl and Ruci\'{n}ski~\cite{rodl1995threshold} which we restate in our notation for convenience.
Here $G(n,m)$ is the Erd\H{o}s-R\'{e}nyi random graph, obtained by selecting a graph from all graphs with $n$ vertices and $m$ edges uniformly at random. 
\begin{prop}[Corollary $4$(c) in ~\cite{rodl1995threshold}] \label{prop:rodl1995threshold}
    Let $r\geq 2$ and let $H$ be a graph with a cycle. Then there exist a constant $C$ such that for $m=Cn^{2-1/m_2(G)}$ and any $t\geq 0$ the random graph $G(n,m)$ contains a subgraph $F'$ that satisfies the following with probability tending to $1$ as $n$ tends to infinity:
    \begin{itemize}
        \item Every subgraph $G\subseteq F'$ with $v(G)\leq t$ satisfies $m_2(G)\leq m_2(H)$
        \item $F'\rightarrow_r H$
    \end{itemize}  
\end{prop}

For this proof it suffices that for our graph $H$, $r=2$ and any $t\geq 1$ there exists at least one graph $F'$ satisfying both conditions of Proposition~\ref{prop:rodl1995threshold}.
By Proposition~\ref{prop:stealing} it follows that $F' \MBarrow H$.
We show that on every subgraph $G\subseteq F'$ with $v(G)\leq t$,
Breaker wins the $H$-game. It then follows that every minimal winnable graph for $H$
contained in $F'$ needs to have more than $t$ vertices, and hence $|\mathcal{M}_{MB}(H)|=\infty$, as $t$ is arbitrary.

From now on let $G\subseteq F'$ with $v(G)\leq t$, and note that $m_2(G) \leq m_2(H)$.
Let $H'\subseteq H$ be a minimal subgraph such that $m_2(H) = \frac{e(H')-1}{v(H')-2}$,
and note that $e(H')\geq 3$, because $m_2(H)>1$ and hence $H'$ needs to contain a cycle. Then
consider the hypergraph $\mathcal{G}$ whose vertices are the edges of $G$,
and where each copy of $H'$ in $G$ forms a hyperedge. We define a cycle in a hypergraph to be a cyclic sequence of vertices $(v_0,v_1,\ldots,v_{\ell-1})$ with $\ell \geq 2$ such that any two consecutive vertices
share a distinct hyperedge.
(Note that in case $\ell=2$ this requires two hyperedges
that contain both $v_0$ and $v_1$.) Then the following claim holds:

\begin{clm}\label{clm:no.Hcycle}
$\mathcal{G}$ is acyclic.
\end{clm}

In fact, the above claim is mentioned within the text of~\cite{rodl1995threshold}, yet no proof is given there. 
For completeness we provide an argument here. 

\begin{proof}[Proof of Claim~\ref{clm:no.Hcycle}]
For contradiction, assume that the claim is wrong. Let
$(v_0,\ldots,v_{\ell-1})$ be a shortest cycle in $\mathcal{G}$,
and let
$e_0,\ldots,e_{\ell-1}$ be distinct 
hyperedges such that
$v_i,v_{i+1}\in e_i$ for every $0\leq i\leq \ell -1$ (indices taken modulo $\ell$).
Furthermore, let $H'_i$ be the copy of $H'$ represented by $e_i$ for every $0\leq i \leq \ell -1$, and
let $G'=H'_0\cup H'_1\cup \ldots \cup H'_{\ell -1}$.

Assume first that $\ell \geq 3$. Then, by minimality of the cycle, we have $|E(H'_i) \cap E(H'_j)| = 1$ and $|V(H'_i) \cap V(H'_j)| \geq 2$ if  $e_i$ and $e_j$ are consecutive hyperedges of the cycle, and $|E(H'_i) \cap E(H'_j)| = 0$ and $|V(H'_i) \cap V(H'_j)| \geq 0$ otherwise. It follows that
\begin{align*}
m_2(G') 
\geq \frac{e(G') -1}{v(G')-2}
\geq \frac{\ell e(H') - \ell - 1}{\ell v(H') - 2\ell - 2} 
> \frac{e(H') - 1}{v(H') - 2} = m_2(H),
\end{align*}
and hence $m_2(G) > m_2(H)$,
a contradiction.

Assume then that $\ell = 2$. Let $\tilde{H} = H'_0 \cap H'_1$. Then
\begin{align*}
m_2(G') 
\geq \frac{e(G') -1}{v(G')-2}
= \frac{2e(H') - e(\tilde{H}) - 1}{2v(H') - v(\tilde{H}) - 2} 
> \frac{e(H') - 1}{v(H') - 2}
\end{align*}
since the last inequality is equivalent to $\frac{e(H')-1}{v(H')-2} > \frac{e(\tilde{H})-1}{v(\tilde{H})-2}$,
which holds by the minimality of $H'$. Again we get $m_2(G) > m_2(H)$, a contradiction.
\end{proof}

Now, from Claim~\ref{clm:no.Hcycle} it immediately follows that all copies of $H'$ in $G$
can be ordered, say $H'_1,\ldots,H'_k$, 
in such a way that $H'_i$ has at most one edge in common with $\bigcup_{j<i} H'_j$.
Since $e(H'_i)\geq 3$, we can find two edges 
$$e_1^{(i)},e_2^{(i)} \in E(H'_i) \setminus \bigcup_{j<i} E(H'_j) .$$
Breaker simply plays according to a pairing strategy with pairs $(e_1^{(i)},e_2^{(i)})$,
$i\in [k]$. This way he can prevent Maker from claiming a copy of $H'$ in $G$,
which also means that Breaker blocks all copies of $H$ in  the game on $G$.
\end{proof}

Finally, let us consider graphs that are not connected, but contain at least one component with three or more edges.

\begin{lemma}\label{Non-connected_graphs}
    Let $F$ be a non-connected graph, such that at least one component of $F$ has at least three edges. Then $\mathcal{M}_{MB}(F)$ and $\mathcal{M}_{BM}(F)$ both consist of infinitely many graphs.
\end{lemma}

\begin{proof}
Let $t\in \mathbb{N}$ be arbitrary. In order to show that
$\mathcal{M}_{MB}(F)$ and $\mathcal{M}_{BM}(F)$ consists of infinitely many graphs, we show that
there is a minimal MB-winnable graph and a minimal BM-winnable graph for $F$ which has at least $t$ vertices. In order to do so, we first construct a BM-winnable graph for $F$ as follows.

Let $\mathcal{H}_1$ be the family of components of $F$ (with multiplicity).
Then as long as $\mathcal{H}_i$ contains a graph with at least three edges, do the following iteratively:
let $H_i\in \mathcal{H}_i$ be an arbitrary component with $e(H_i)\geq 3$,
and fix an arbitrary graph $G_i\in \mathcal{M}_{MB}(H_i)$ such that $v(G_i)\geq t$.
Note that $G_i$ exists by Lemma~\ref{lem:trees} and Lemma~\ref{cycles}.
Afterwards, set $\mathcal{H}_i^\ast := \{H\in \mathcal{H}_i: G_i\MBarrow H\}$, 
	$x_i := |\mathcal{H}_i^\ast|$, and $\mathcal{H}_{i+1} := \mathcal{H}_i\setminus \mathcal{H}_i^\ast$.
Assume the iteration takes $s$ rounds, so that $\mathcal{H}_{s+1}$ contains only
graphs with at most two edges. Then set $r:=|\mathcal{H}_{s+1}|$ and consider the graph
$G := 2x_1 G_1 + 2 x_2 G_2 + \ldots + 2 x_s G_s + r P_7$. The following two claims hold.

\begin{clm} \label{clm:forests:F}
$G \BMarrow F$.
\end{clm}

\begin{proof}
For each $i\in [s]$, Maker distributes the copies of $G_i$ into $x_i$ pairs.
Then Maker considers each of these pairs and the subgraph $r P_7$ as a separate subboard.
That is, when Breaker makes a move on one of these subboards, Maker also claims an edge of the same board.

For $i\in [s]$, consider a pair of $G_i$-copies. When Breaker claims a first edge in one of these copies, Maker can claim a first edge in the other copy, hence pretending to be the first player on this copy. 
Hence, on half of the copies of $G_i$ Maker can ensure to play as first player,
and since $G_i \MBarrow H$ holds for every $H\in \mathcal{H}_i^\ast$, Maker can thus claim 
a copy of $H$ for every $H\in \mathcal{H}_i^\ast$. 
The remaining $r$ components of $F$, i.e.~those in $\mathcal{H}_{s+1}$,
have at most two edges each. A simple case distinction shows that $P_7 \BMarrow P_3$ holds,
and hence Maker can use $r P_7$ in order to finish a copy of $F$.
\end{proof}

\begin{clm} \label{clm:forests:notF}
Let $S$ contain an edge of each of the $2x_s$ copies of $G_s$. Then $G - S \MBarrow F$ does not hold. 
\end{clm}

\begin{proof}
By definition we have $H_s \in \mathcal{H}_s$ and hence $H_s \notin \bigcup_{i < s} \mathcal{H}_i$.
In particular, we know $G_i \MBnotarrow H_s$ for every $i<s$.
Moreover, since $G_s\in \mathcal{M}_{MB}(H_s)$,
we know that $G_s - e \MBnotarrow H_s$ for every $e\in E(G_s)$.
Additionally, a simple case distinction shows that $P_7 \MBarrow H$ does not hold
for graphs $H$ with at least $3$ edges, and hence 
$P_7 \MBarrow H_s$ does not hold.
Thus, on each of the components of $G - S$, Breaker has a strategy to block
all copies of $H_s$. Breaker simply applies these strategies in parallel by always playing on the component that Maker played on in her previous move. This way, Maker cannot get a copy of $H_s$ and hence loses the $F$-game on $G-S$ as first player.
\end{proof}

Using the above two claims we can finish the main proof. Because of Claim~\ref{clm:forests:F}, we know $G \BMarrow F$ and hence $G \MBarrow F$. In particular $G$ contains a graph $G' \in \mathcal{M}_{BM}(F)$ and a graph $G'' \in \mathcal{M}_{MB}(F)$.
Moreover by Claim~\ref{clm:forests:notF} the graphs $G''$ and also $G'$ both need to contain a copy of $G_s$ and thus $v(G'')$ and $v(G')$ are both at least $v(G_s)\geq  t$.
\end{proof}
\medskip

\subsection{Proof of Theorem~\ref{finiteVsInfinite}}
We quickly summarize how 
Theorem~\ref{finiteVsInfinite} follows from our previous discussion.
When $H$ is a connected graph with at least three edges, $\mathcal{M}_{MB}(H)$ consists of infinitely many graphs by Lemma~\ref{lem:trees} or Lemma~\ref{cycles}. By Lemma~\ref{lem:MBBM-Lemma} this means that for those graphs also $\mathcal{M}_{BM}(H)$ consists of infinitely many graphs. Finally, Lemma~\ref{lem:finite:matching.and.cherries} and Lemma~\ref{Non-connected_graphs} show that the equivalence stated in Theorem~\ref{finiteVsInfinite} also holds for all non-connected graphs.\hfill $\Box$

\medskip

\section{Clique number and chromatic number} \label{sec:clique.chromatic}

The following proof is a generalization of an argument in~\cite{nevsetvril2010ramsey}.

\begin{proof}[Proof of Theorem~\ref{thm:construction.chromatic.clique}]
Let $H$ and $b$ be given.
Let $N > b^2 (1+b)^{e(H)-3} e(H)$
be an arbitrary integer, and let $G_0$ be the $N$-blow-up of $H$,~i.e.~replace 
the vertices $v_i\in V(H)$ by disjoint sets $V_i$ of size $N$, with $i\in [v(H)]$, 
and add all edges between two sets $V_i$ and $V_j$ if $v_iv_j\in E(H)$.
Then $\omega(G_0)=\omega(H)$ and $\chi(G_0)=\chi(H)$.

Next, let $(X,\mathcal{F})$ be the following hypergraph:
$X$ is the edge-set of $G_0$, and $\mathcal{F}$ consists of all copies of $H$
for which the copy of $v_i$ belongs to $V_i$. Then
$|X|=N^2 e(H)$,
$|\mathcal{F}| = N^{v(H)}$,
$\Delta_2(\mathcal{F}) = N^{v(H)-3}$,
and $|F|=e(H)$ for every $F\in \mathcal{F}$.
Therefore, using the choice of $N$, we conclude
\begin{align*}
\sum_{F\in \mathcal{F}} (1+b)^{-|F|} 
=N^{v(H)} (1+b)^{-e(H)}
> \frac{b^2}{(1+b)^3} N^{v(H)-1} e(H)
= \frac{b^2}{(1+b)^3} \Delta_2(\mathcal{F}) |X| .
\end{align*}
Thus, by Theorem~\ref{thm:Weak.Win} we know that
Maker wins the $(1:b)$ Maker-Breaker game on $(X,\mathcal{F})$,
and hence $G_0 \MBarrow_b H$.
In particular, $G_0$ needs to contain a minimal MB-winnable
graph for $H$ with bias $b$ as a subgraph; call it $G$.
Then $\omega(H) \leq \omega(G) \leq \omega(G_0) = \omega(H)$, and hence
$\omega(G)= \omega(H)$. Analgously, we have $\chi(G)=\chi(H)$.
\end{proof}

\medskip

\section{Minimum degrees} \label{sec:degrees}

In this section we first prove Theorem~\ref{thm:general.bounds},
and later on Theorem~\ref{thm:some.precise.degrees}. In fact, for the upper bounds 
in Theorem~\ref{thm:general.bounds} we prove the following stronger result.

\begin{thm}\label{thm:general.bound.Delta}
Let $H$ be a graph with at least one edge.
Let $c:V(H)\rightarrow [\chi(H)]$ be a proper vertex colouring,
and let $\Delta_c(H) = \max\{d_H(v):~v\in V(H),~ c(v)=1\}$.
Then 
$$
s^b_{MB}(H) \leq 
\begin{cases}
2\Delta_c(H) - 1 & ~~ \text{if }b=1 \\
b\Delta_c(H) (\ln(\Delta_c(H)) + 2) & ~~ \text{if }b \geq 2 . \\
\end{cases}
$$
\end{thm}

\subsection{Preparation}

Before we can prove Theorem~\ref{thm:general.bound.Delta}
and Theorem~\ref{thm:general.bounds},
we introduce some useful notation and state a suitable embedding lemma
for our later analysis of Maker's strategy.

\begin{definition}
Given positive integers $k,t,d$ with $t \geq d\geq 2$, we define
$H_{k,t,d}$ to be the $d$-uniform hypergraph
with a partition $V = V_1\cup \ldots \cup V_t$ of its vertex set such that $|V_i|=k$ for every $i\in [t]$,
and such that the  edge set equals 
$$E(H_{k,t,d}) = \left\{ e\in \binom{V}{d}:~ (\forall i\in[t]:~ |e\cap V_i| \leq 1)\right\}.$$ 
\end{definition}

\begin{definition} \label{def:family.Fktd}
Given positive integers $k,t,d,s$ with $t\geq d\geq 2$ and $k\geq s$
we define $\mathcal{F}_s(k,t,d)$ to be the family of all $d$-uniform hypergraphs
$\mathcal{H}$ that satisfy the following property:
There exists a partition of the vertex set $V(\mathcal{H}) = V_1\cup \ldots \cup V_t$
such that
\begin{enumerate}
\item[(i)] $|V_i| = k$ for every $i\in [t]$,
\item[(ii)] $|e\cap V_i| \leq 1$ for every $i\in [t]$ and $e\in E(\mathcal{H})$,
\item[(iii)] for every $i\in [t]$, every vertex $v\in V_i$, 
every index set $J\in \binom{[t]\setminus [i]}{d-1}$ and every choice of sets
$A_j\in \binom{V_j}{s}$ for $j\in J$, we have 
$e_{\mathcal{H}}\big( \{v\}\cup \bigcup_{j\in J} A_j\big) \neq 0.$
\end{enumerate}
\end{definition}

We make use of the following embedding result.

\begin{lemma}\label{lem:embedding.ktd}
For all positive integers $k,t,d,s$ with $t\geq d \geq 2$
there exists a positive integer $K = K(k,t,d,s)$ such that
every hypergraph in $\mathcal{F}_s(K,t,d)$ contains a copy of
$H_{k,t,d}$.
\end{lemma}

We postpone the proof of Lemma~\ref{lem:embedding.ktd} to Section~\ref{sec:embedding.ktd},
and first show how Theorem~\ref{thm:general.bound.Delta} and Theorem~\ref{thm:general.bounds} 
can be deduced from that.

\subsection{Proof of general bounds}

\begin{proof}[Proof of Theorem~\ref{thm:general.bound.Delta}]
Let $c:V(H)\rightarrow [\chi(H)]$ be a proper vertex colouring,
let $m:=|c^{-1}(1)|$ and let $H' = H - c^{-1}(1)$
with vertex set $V(H')=\{v_1,\ldots,v_{v(H')}\}$.
Let 
$$
s := 
\begin{cases}
2 & ~~ \text{if }b=1 \\
\left\lceil b \big(\ln (\Delta_c(H)) + 1 \big) \right\rceil & ~~ \text{if }b \geq 2 ,
\end{cases}
$$
and let
$K=K(1,v(H'),\Delta_c(H),s)$ be the constant promised by
Lemma~\ref{lem:embedding.ktd} (i.e.~with input $k:=1$, $t:=v(H')$, $d:=\Delta_c(H)$ and $s$).
Moreover, let $\widetilde{H}$ be the $K$-blow-up of $H'$,~i.e.~replace 
the vertices $v_i\in V(H')$ by disjoint sets $V_i$ of size $K$, with $i\in [v(H')]$, 
and add all edges between two sets $V_i$ and $V_j$ if $v_iv_j\in E(H')$.
Then $\chi(\widetilde{H}) = \chi(H') = \chi(H) - 1$.
We now consider the following contruction:
\begin{enumerate}
\item[(1)] Let $G_0$ be a graph such that $G_0\in \mathcal{M}_{MB}^b(\widetilde{H})$ and
$\chi(G_0)=\chi(\widetilde{H})$. (existence by Theorem~\ref{thm:construction.chromatic.clique})
\item[(2)] For every vertex set $S\subset V(G_0)$ of size $s(\Delta_c(H)-1)+1$ add a set $W_S$
of $$(b+1) \cdot \left[ e(G_0) + \binom{v(G_0)}{s(\Delta_c(H)-1) + 1} \Delta_c(H) ms^{\Delta_c(H)-1} \right]$$ new vertices, and add all edges between $S$ and $W_S$.
\end{enumerate}
Call the resulting graph $G$ so that we obtain the vertex set
$V(G)  := V(G_0) \cup \bigcup_{S} W_S$ and the edge set $E(G)  := E(G_0) \cup \bigcup_{S} \{ xy:~ x\in S, y\in W_S\}$,
where the union runs over all $S\subset V(G_0)$ of size $s(\Delta_c(H)-1)+1$.
We claim that the following holds: 
\begin{enumerate}
\item[(a)] $G_0$ is not winnable for $H$, and  (b) $G$ is winnable for $H$.
\end{enumerate}

Note that once (a) and (b) are proved, the theorem follows. Indeed,
by (b) we then know that $G$ needs to have a minimal winnable graph for $H$ as subgraph, say $G'$. 
Then, by (a) it follows that $G'$ needs to use at least one vertex $x\in V(G)\setminus V(G_0)$. 
Hence, $\delta(G') \leq d_{G'}(x) \leq d_G(x) = s(\Delta_c(H)-1)+1$, and the desired bound on $s_{MB}^b(H)$ follows.

\medskip

Property (a) is trivial, since 
$\chi(G_0) < \chi(H)$, and hence
$G_0$ does not contain a copy of $H$. Hence, it remains to prove (b), 
i.e.~to find a winning strategy for Maker in the $H$-game on $G$.
In the following, we first describe a suitable strategy and explain why 
Maker can follow it (Strategy). Only afterwards, we explain why she claims a copy of $H$ 
by following the strategy (Embedding part).

\medskip

\textbf{Strategy:} Maker's strategy consists of two stages.

\smallskip

\textbf{Stage I:} At first, Maker plays on $G_0$ only and claims a copy of $\widetilde{H}$. This she can do by the choice of $G_0$.

\smallskip

\textbf{Stage II:} Abusing notation slightly, denote this copy of $\widetilde{H}$ by $\widetilde{H}$,
and as before, let $V_1,\ldots,V_{v(H')}$ denote the vertex subsets in $\widetilde{H}$
representing the blow-ups of the vertices $v_1,\ldots,v_{v(H')}\in V(H')$. 
Now, Maker considers iteratively (in an arbitrary ordering) 
all choices of
an index $i\in [v(H')]$, 
a vertex $v\in V_i$, 
an index set $J\in \binom{[v(H')]\setminus [i]}{\Delta_c(H)-1}$, 
subsets $A_j \in \binom{V_j}{s}$ with $j\in J$,
and the set $S:=\{v\}\cup \bigcup_{j\in J} A_j$,
and for each such a choice she does the following:
within at most $\Delta_c(H) ms^{\Delta_c(H)-1}$ rounds,
she ensures that for $m s^{\Delta_c(H)-1}$
vertices $w\in W_S$
she claims the edge $wv$ and one edge
from $w$ to $A_j$ for every $j\in J$.

She can do this by the following reason:
At first observe that, in
Stage I and II together Maker plays at most
$$e(G_0) + \binom{v(G_0)}{s(\Delta_c(H)-1) + 1} \Delta_c(H) m s^{\Delta_c(H)-1}$$
rounds. 
Indeed, Stage I is played on $G_0$ only and this takes at most
$e(G_0)$ rounds.
Moreover, there are less than 
$\binom{v(G_0)}{s(\Delta_c(H)-1) + 1}$ iterations in Stage II and each of them
lasts at most $\Delta_c(H) m s^{\Delta_c(H)-1}$ rounds.

Hence, by part (2) of the construction, at any moment in Stage II and for any choice of $i,v,J,A_j,S$ 
as described above, Maker finds a vertex $w\in W_S$ such that all edges
between $w$ and $S$ are still free.
Maker can then claim $wv$ first, and afterwards, applying a pairing strategy (if $b=1$)
or BoxBreaker's strategy from Theorem~\ref{thm:Box.game} (if $b\geq 2$),
she can get exactly one edge from $w$ to $A_j$ for every $j\in J$.
As this takes $\Delta_c(H)$ rounds for any such fixed vertex $w$,
an iteration for a fixed choice of $i,v,J,A_j,S$
lasts at most $\Delta_c(H) m s^{\Delta_c(H)-1}$ rounds, as required.

\medskip

\textbf{Embedding part:} It remains to show that Maker gets a copy of $H$.

\smallskip

Let $i,v,J,A_j,S$ be any choice as described in Stage II.
Maker has $m s^{\Delta_c(H)-1}$ vertices $w\in W_S$
such that she claims the edge $wv$ and one edge
from $w$ to $A_j$ for every $j\in J$.
Since there are only $s^{|J|}$ options how the neighbours of a fixed vertex $w\in W_S$ 
in these sets $A_j$ may be chosen, by the pigeonhole principle,
Maker then finds $m$ vertices in $W_S$ with the same neighbours.
In particular, there is a set $S'\subset S$ of size $\Delta_c(H)$, which contains $v$ and one vertex
from each $A_j$, $j\in J$, such that Maker has all edges between these $m$ vertices of $W_S$ 
and the set $S'$. For the given choice of $i,v,J,A_j,S$ fix one such set $S'$ and give it colour red.

Running over all choices $i,v,J,A_j,S$,
the resulting hypergraph $\mathcal{R}$ of red hyperedges lives on $\bigcup_{i\in v(H')} V_i$
and, due to its construction, it belongs to the family $\mathcal{F}_s(K,v(H'),\Delta_c(H))$.
Hence, by Lemma~\ref{lem:embedding.ktd} and the choice of $K$,
we find a copy of $H_{1,v(H'),\Delta_c(H)}$ in this hypergraph.
Let $x_i\in V_i$ be the unique vertex belonging to this copy, for every $i\in [v(H')]$.
Then, since $\widetilde{H}$ is a blow-up of $H'$, the vertex $x_i$ can be considered as a copy of $v_i\in V(H')$, i.e.~on the vertex set $\{x_1,\ldots,x_{v(H')}\}$
we find a copy $F'$ of $H'$. Moreover, since every set $S'\in \binom{V(F')}{\Delta_c(H)}$ 
is a red hyperedge in $\mathcal{R}$, we know that for every such set $S'$
there exist at least $m$ additional vertices such that Maker claims all edges between them and
$S'$. Using these additional vertices, we can complete $F'$ to a copy of $H$,
since every vertex in $V(H)\setminus V(H')$ has degree at most $\Delta_c(H)$
and all its neighbours belong to $V(H')$. 
\end{proof}

\begin{proof}[Proof of Theorem~\ref{thm:general.bounds}]
We first prove that $(b+1)(\delta(H)-1) + 1 \leq s^b_{BM}(H)$ for every $b\in \mathbb{N}$,
hence giving the first inequalities in (i) and (ii) of the theorem.
For contradiction, assume that there exists $G\in \mathcal{M}_{BM}^b(H)$ 
such that $\delta(G) \leq (b+1)(\delta(H)-1)$.
Then there exists a vertex $v\in V(G)$ with $d_G(v)\leq (b+1)(\delta(H)-1)$.
Since $G$ is minimal, we know that Breaker has a strategy $\mathcal{S}$ to always 
win the $(1:b)$ biased $H$-game on $G-v$ (when Breaker is the first player).
On the graph $G$, Breaker can play on $G-v$ according to strategy $\mathcal{S}$ (being the first player)
and from the remaining edges incident with $v$ he can always claim (up to) $b$ edges whenever Maker claimed an edge incident with $v$ in the previous rounds.
This way, Maker cannot obtain a copy of $H$ within $G-v$ (as Breaker plays according to $\mathcal{S}$)
and Maker also cannot claim a copy of $H$ involving the vertex $v$,
as Maker gets at most $\delta(H)-1$ edges incident with $v$. Hence, Breaker 
wins the $(1:b)$ biased $H$-game on $G-v$, in contradiction to $G\in \mathcal{M}_{BM}^b(H)$.

\medskip

Next, let us prove $s_{BM}^b(H) \leq s_{MB}^b(H)$,
hence giving the second inequalities in (i) and (ii) of the theorem.
For this let $F\in \mathcal{M}_{MB}^b(H)$ with $\delta(F) = s_{MB}^b(H) =:k$,
and let $v\in V(F)$ be an arbitrary vertex of $F$ with $d_F(v) = k$.
We consider the graph
$G:=F_1 + \ldots + F_{b+1}$
where each $F_i$ is a copy of $F$,
and we let $v_i$ denote the copy of $v$ in $F_i$, for every $i\in [b+1]$. 
Then, even if Maker plays as second player on $G$, she can ensure to be the first player on
one of the copies $F_i$. Playing on this copy only, Maker can win the $H$-game,
since $F_i\in \mathcal{M}_{MB}^b(H)$.
Moreover, independent of whether Maker is first or second player, 
she cannot win the $(1:b)$ biased $H$-game on $G-\{v_i:~i\in [b+1]\}$,
since Breaker (as second player) wins the $(1:b)$ biased $H$-game on $F_i-v_i$, for every $i\in [b+1]$, and since $H$ is connected.
Hence, $G$ needs to contain a subgraph $G'\in \mathcal{M}_{BM}^b(H)$,
but this subgraph needs to contain at least one of the vertices $v_i$.
Therefore, $s_{BM}^b(H) \leq d_{G'}(v_i) \leq d_{F_i}(v_i) = k = s_{MB}(H)$.

\medskip

Finally, the third inequalities in (i) and (ii) of the theorem follow directly from    
Theorem~\ref{thm:general.bound.Delta}.
\end{proof}

\subsection{Proof of Lemma~\ref{lem:embedding.ktd}} \label{sec:embedding.ktd}

For the proof of Lemma~\ref{lem:embedding.ktd} we make use of the following generalization of
the K\H{o}v\'ari-S\'os-Tur\'an bound~\cite{kovari1954problem}.

\begin{thm}[Theorem 1 in~\cite{erdos1964extremal}] \label{thm:erdos.extremal}
For every $k,d\in\mathbb{N}$ there exists $n_0\in \mathbb{N}$
such that for all integers $n\geq n_0$ the following holds:
Every $d$-uniform hypergraph $\mathcal{H}$ on $n$ vertices and
more than $n^{d-k^{1-d}}$ edges contains a copy of $H_{k,d,d}$.
\end{thm}

The proof of Lemma~\ref{lem:embedding.ktd} is done by induction,
where the base case follows almost directly from the above theorem.
In order to make the inductive step work, we extend Definition~\ref{def:family.Fktd}.

\begin{definition}
Given a $d$-uniform hypergraph $\mathcal{H}$,
a vertex $v\in V(\mathcal{H})$ and a vertex set
$A\subseteq V(\mathcal{H})\setminus \{v\}$ we let
the $(d-1)$-uniform hypergraph
$$
\mathcal{H}_v[A] := \left( A, \left\{ f\in \binom{A}{d-1}:~ f\cup \{v\}\in E(\mathcal{H})\right\} \right)
$$
be the link of $v$ in $A$.
\end{definition}

\begin{definition}\label{def:FktdAJs}
Let $k,t,d,s,r$ be positive integers with $t\geq d \geq 2$ and $k\geq r$,
let $\mathcal{A} \subseteq \binom{[t]\setminus \{1\}}{d-1}$
and $\mathcal{K}\in \binom{[t]\setminus \{1\}}{d-1} \setminus \mathcal{A}$.
We define $\mathcal{F}_s(k,t,d,\mathcal{A},\mathcal{K},r)$ to be the family of all 
$d$-uniform hypergraphs
$\mathcal{H}$ that satisfy the following property:
There exists a partition of the vertex set $V(\mathcal{H}) = V_1\cup \ldots \cup V_t$
such that
\begin{enumerate}
\item[(i)] $|V_i| = k$ for every $i\in [t]$,
\item[(ii)] $|e\cap V_i| \leq 1$ for every $i\in [t]$ and $e\in E(\mathcal{H})$,
\item[(iii)] for every $i\in [t]$, every vertex $v\in V_i$, 
every index set $J\in \binom{[t]\setminus [i]}{d-1}$ and every choice of sets
$A_j\in \binom{V_j}{s}$ for $j\in J$, we have 
$e_{\mathcal{H}}\big( \{v\}\cup \bigcup_{j\in J} A_j\big) \neq 0$,
\item[(iv)] for every $v\in V_1$ and every $J\in \mathcal{A}$ we have 
	$\mathcal{H}_v\big[ \bigcup_{j\in J} V_j \big] \cong H_{k,d-1,d-1}$,
\item[(v)] there exists a set $X_1 \in \binom{V_1}{r}$ such that 
	 $\mathcal{H}_v\big[ \bigcup_{j\in \mathcal{K}} V_j \big] \cong H_{k,d-1,d-1}$ for every $v\in X_1$.
\end{enumerate}
\end{definition}

Next, we prepare our inductive proof with three short lemmas.

\begin{lemma}\label{lem:FktdAJs_part1}
For all positive integers $k,t,d,s$ with $t\geq d \geq 2$ there exists an integer $K= f_{t,d,s}(k)$
such that the following holds for every 
$r<k$, every family $\mathcal{A} \subset \binom{[t]\setminus \{1\}}{d-1}$
and every index set $\mathcal{K}\in \binom{[t]\setminus \{1\}}{d-1} \setminus \mathcal{A}$:
Every hypergraph $\mathcal{H}\in \mathcal{F}_s(K,t,d,\mathcal{A},\mathcal{K},r)$
contains a hypergraph $\mathcal{H}'\in \mathcal{F}_s(k,t,d,\mathcal{A},\mathcal{K},r+1)$.
\end{lemma}

\begin{proof}
We choose $K$ to be large.
For $\mathcal{H}\in \mathcal{F}_s(K,t,d,\mathcal{A},\mathcal{K},r)$,
let $V(\mathcal{H})=V_1\cup \ldots \cup V_t$ be the partition promised by
Definition~\ref{def:FktdAJs}, and let $X_1$ be the set promised by (v).
Choose $v\in V_1\setminus X_1$ (which is possible for $K\geq k$).
Then from (iii) we can conclude that
the $(d-1)$-uniform hypergraph $\mathcal{H}_v := \mathcal{H}_v \big[ \bigcup_{j\in \mathcal{K}} V_j \big]$
satisfies
$$
e\left( \mathcal{H}_v \right) \geq \frac{1}{s^{d-1}} K^{d-1} 
= \frac{1}{s^{d-1}(d-1)^{d-1}} \cdot v(\mathcal{H}_v)^{d-1}.
$$
One way to see the first equality is to use the following random experiment. At first choose uniformly at random sets $A_j \in \binom{V_j}{s}$ for $j\in \mathcal{K}$, and then choose uniformly at random a vertex $a_j \in A_j$ for $j\in \mathcal{K}$. Then
$$
\frac{e\left( \mathcal{H}_v \right)}{K^{d-1}} = \text{Prob}\left( \{a_j:~j\in \mathcal{K}\} \in E(\mathcal{H}_v) \right) \geq \frac{1}{s^{d-1}},
$$
where the equation holds due to the uniform choice, and the inequality holds by property (iii)
applied with $i=1$ and $J=\mathcal{K}$.

By Theorem~\ref{thm:erdos.extremal} 
we know that, if $K$ is large enough and hence $v(\mathcal{H}_v)$ is large enough, 
we can find a subhypergraph $\mathcal{H}_0$ in $\mathcal{H}_v[ \bigcup_{j\in \mathcal{K}} V_j ]$
which is isomorphic to $H_{k,d-1,d-1}$.
For every $j\in \mathcal{K}$ set $V_j' = V(\mathcal{H}_0) \cap V_j$,
for every $j\in [t]\setminus (\mathcal{K}\cup \{1\})$ let $V_j' \in \binom{V_j}{k}$ be arbitrary,
and choose $V_1'\in \binom{V_1}{k}$ such that $X_1\cup \{v\} \subseteq V_1'$.
Then $\mathcal{H}' = \mathcal{H}[ \bigcup_{j\in [t]} V_j']$
belongs to $\mathcal{F}_s(k,t,d,\mathcal{A},\mathcal{K},r+1)$.
\end{proof}

\begin{lemma}\label{lem:FktdAJs_part2}
For all positive integers $k,t,d,s$ with $t\geq d \geq 2$ there exists an integer $K= g_{t,d,s}(k)$
such that the following holds for every 
$\mathcal{A} \subset \binom{[t]\setminus \{1\}}{d-1}$
and every index set $\mathcal{K}\in \binom{[t]\setminus \{1\}}{d-1} \setminus \mathcal{A}$:
Every hypergraph $\mathcal{H}\in \mathcal{F}_s(K,t,d,\mathcal{A},\varnothing,0)$
contains a hypergraph $\mathcal{H}'\in \mathcal{F}_s(k,t,d,\mathcal{A}\cup \{\mathcal{K}\},\varnothing,0)$.
\end{lemma}

\begin{proof}
Let $f_{t,d,s}$ be the function promised by Lemma~\ref{lem:FktdAJs_part1},
and set $K = f_{t,d,s}^k(k)$.
Let $\mathcal{H}_0 := \mathcal{H}\in \mathcal{F}_s(K,t,d,\mathcal{A},\varnothing,0)$,
and hence $\mathcal{H}_0 := \mathcal{H}\in \mathcal{F}_s(K,t,d,\mathcal{A},\mathcal{K},0)$ by definition.
Then applying Lemma~\ref{lem:FktdAJs_part1} iteratively,
we find a sequence 
$\mathcal{H}_0 \supseteq \mathcal{H}_1 \supseteq \ldots \supseteq \mathcal{H}_k$
such that $\mathcal{H}_i \in \mathcal{F}_s(f_{t,d,s}^{k-i}(k),t,d,\mathcal{A},\mathcal{K},i)$
for every $i\in [k]$.
Hence, if we set $\mathcal{H}':=\mathcal{H}_k$, then we have
$\mathcal{H}' \in \mathcal{F}_s(k,t,d,\mathcal{A},\mathcal{K},k) = \mathcal{F}_s(k,t,d,\mathcal{A}\cup \{\mathcal{K}\},\varnothing,0)$.
\end{proof}

\begin{lemma}\label{lem:FktdAJs_part3}
For all positive integers $k,t,d,s$ with $t\geq d \geq 2$ there exists an integer $K= h_{t,d,s}(k)$
such that the following holds:
Every hypergraph $\mathcal{H}\in \mathcal{F}_s(K,t,d,\varnothing,\varnothing,0)$
contains a hypergraph $\mathcal{H}'\in \mathcal{F}_s(k,t,d,\binom{[t]\setminus \{1\}}{d-1},\varnothing,0)$.
\end{lemma}

\begin{proof}
Let $q:=\binom{t-1}{d-1}$.
Let $g_{t,d,s}$ be the function promised by Lemma~\ref{lem:FktdAJs_part2},
and set $K = g_{t,d,s}^q(k)$.
Let $\mathcal{H}_0 := \mathcal{H}\in \mathcal{F}_s(K,t,d,\varnothing,\varnothing,0)$.
Moreover, let
$$
\varnothing \subset \mathcal{A}_1 \subset \mathcal{A}_2 \subset \ldots \subset \mathcal{A}_q 
= \binom{[t]\setminus \{1\}}{d-1}
$$
be a maximal chain, i.e.~$|\mathcal{A}_i|=i$ for every $i\in [q]$.
Then applying Lemma~\ref{lem:FktdAJs_part2} iteratively,
we find a sequence 
$\mathcal{H}_0 \supseteq \mathcal{H}_1 \supseteq \ldots \supseteq \mathcal{H}_q$
such that $\mathcal{H}_i \in \mathcal{F}_s(g_{t,d}^{q-i}(k),t,d,\mathcal{A}_i,\varnothing,0)$
for every $i\in [q]$.
Hence, if we set $\mathcal{H}':=\mathcal{H}_q$, then we have
$\mathcal{H}' \in \mathcal{F}_s(k,t,d,\mathcal{A}_q,\varnothing,0)$, as desired.
\end{proof}

Now we are ready to prove the main lemma of this subsection.

\begin{proof}[Proof of Lemma~\ref{lem:embedding.ktd}]
Let $k,d,s$ be fixed. We prove the lemma by induction on $t\geq d$.

For the base case, let $t=d$. Let $\mathcal{H}\in \mathcal{F}_s(K,d,d)$.
Then by properties (i) and (iii), we have
$$
e\left( \mathcal{H} \right) 
= \sum_{v\in V_1} e\left( \mathcal{H}_v[V\setminus V_1] \right) 
\geq \sum_{v\in V_1} \frac{1}{s^{d-1}} K^{d-1}
=\frac{1}{s^{d-1}} K^{d} = \frac{1}{s^{d-1}d^d} v(\mathcal{H})^{d}
$$
where the inequality holds by the same argument as in the proof of Lemma~\ref{lem:FktdAJs_part1}.
By Theorem~\ref{thm:erdos.extremal} 
we know that, if $K$ is large enough and hence $v(\mathcal{H})$ is large enough, 
we find a copy of $H_{k,d,d}$ in $\mathcal{H}$.

\medskip

For the induction step, let $t>d$. 
By induction we then know that there exists a constant $K_0$
such that every hypergraph in $\mathcal{F}_s(K_0,t-1,d)$
contains a copy of $H_{k,t-1,d}$. 
Then, let $K=h_{t,d,s}(K_0)$ be the constant promised by Lemma~\ref{lem:FktdAJs_part3}.
Consider any hypergraph $\mathcal{H}\in \mathcal{F}_s(K,t,d)$.
Then by Lemma~\ref{lem:FktdAJs_part3},
we find a subhypergraph 
$\mathcal{H}_0\in \mathcal{F}_s(K_0,t,d,\binom{[t]\setminus \{1\}}{d-1},\varnothing,0)$.
Let $V(\mathcal{H}_0) = V_1\cup \ldots \cup V_t$ be a partition as promised by Definition~\ref{def:FktdAJs}.
Then for every vertex $v\in V_1$ we have 
$(\mathcal{H}_0)_v[V\setminus V_1] \cong H_{K_0,t-1,d-1}$, because of (iv).
Next, let $\mathcal{H}_0' :=\mathcal{H}_0[V\setminus V_1]$ and note that
$\mathcal{H}_0'\in \mathcal{F}_s(K_0,t-1,d,\varnothing,\varnothing,0) = \mathcal{F}_s(K_0,t-1,d)$.
Hence, by induction hypothesis, $\mathcal{H}_0'$ contains a copy $H'$ of 
$H_{k,t-1,d}$. Let its vertex set be $V(H') = V_2' \cup \ldots \cup V_t'$
such that $V_i' \subset V_i$ for every $i\in [t]\setminus \{1\}$.
Additionally, choose $V_1'\subset V_1$ to be any subset of size $k$.
Then $\mathcal{H}[ \bigcup_{i\in [t]} V_i' ] \cong H_{k,t,d}$.
\end{proof}

\subsection{Precise minimum degrees}

\begin{proof}[Proof of Theorem~\ref{thm:some.precise.degrees}]
For (a) note that the graph $K_{s,t}$ has a unique proper $2$-colouring, and all vertices
in its larger vertex class have minimum degree. By Theorem~\ref{thm:general.bound.Delta}
we conclude
$s_{MB}^1(K_{s,t}) \leq 2\min\{s,t\} - 1$, while Theorem~\ref{thm:general.bounds}
gives $2\min\{s,t\} - 1 \leq s_{BM}^1(K_{s,t}) \leq s_{MB}^1(K_{s,t})$.

For (b) note that no matter how the vertices of $W_t$ are coloured properly with $\chi(W_t)$ colours,
there is a colour class which contains only vertices of degree 3. Hence, 
Theorem~\ref{thm:general.bound.Delta} implies
$s_{MB}^1(W_t) \leq 5$. Additionally, Theorem~\ref{thm:general.bounds}
implies $5 \leq s_{BM}^1(W_t) \leq s_{MB}^1(W_t)$.
\end{proof}

\begin{proof}[Proof of Theorem~\ref{thm:some.precise.degrees.biased}]
We start with (a). For every $b\in\mathbb{N}$ and $t\geq 3$, we have 
$s^b_{MB}(C_t) \geq s^b_{BM}(C_t)\geq b+2$ by Theorem~\ref{thm:general.bounds}. Hence, it remains to show that 
$s^b_{MB}(C_t) \leq b+2$. For this, we consider the following simple construction:
Let $k\in\mathbb{N}$ 
such that $t\in \{2k+1,2k+2\}$,
let $q=(b+2)^2k$ and let
$G_0$ by a perfect $q$-are tree
of height $k$ with root $v_0$.
Then for every set $S\subseteq V(G_0)$ of size $b+2$ add a set $A_S$ of $q$ new vertices and also add all possible edges between $S$ and $A_S$; call the resulting graph $G$. As $G_0$ does not contain cycles, we have that 
$G_0 \MBarrow_b C_t$ cannot hold.
As next we show that $G\MBarrow_b C_t$. Then the claimed bound $s^b_{MB}(C_t) \leq b+2$ follows immediately, as every minimal MB-winnable subgraph of $G$
needs to use at least one vertex of one of the sets $A_S$.

In order to show that $G\MBarrow_b C_t$ is true consider the following strategy for Maker: At first, Maker claims only edges of $G_0$ for $(b+2)k$ rounds (details follow). Since in the meantime Breaker cannot claim more than $b(b+2)k$ edges, at most $(b + 1)(b+2)k < q$ edges become claimed in total, and we know that at any moment in these 
$(b+2)k$ rounds, every non-leaf vertex of $G_0$ still has free edges going to some of its children. It thus follows, that within $(b+2)k$ rounds, Maker can easily claim $b+2$ edge-disjoint paths (of length $k$) that all start in $v_0$ and end in a leaf of $G_0$.
Once these paths are claimed, 
denote them with $F_1,\ldots,F_{b+2}$
and let $v_i$ be the endpoint of $F_i$ different from $v_0$, for every $i\in [b+2]$.
We consider two cases for Maker's next moves. If $t$ is even,
then let $S=\{v_1,\ldots,v_{b+2}\}$, 
and note that by construction there exists a set $A_S$ of size $q$ such that there is an edge between every vertex of $S$ and every vertex of $A_S$. Since fewer than $q$ rounds were played so far,
there must be a vertex $w\in A_S$ such that all edges between $S$ and $w$ are still free.
Maker then claims $wv_1$, and in the next round she claims $wv_i$ for some $i\in \{2,\ldots,b+2\}$,
thus finishing a copy of $C_{2k+2}=C_t$.
Otherwise, if $t$ is odd,
Maker basically applies the same strategy, but where in the set $S$ the vertex $v_1$ is replaced with its unique neighbour on the path $F_1$. This way she claims a cycle of length $2k+1$, i.e. a copy of $C_t$.

\medskip

Now let us prove (b).
For every $b\in\mathbb{N}$ and even $t\geq 4$, we have 
$s^b_{MB}(W_t) \geq s^b_{BM}(W_t)\geq 2b+3$ by Theorem~\ref{thm:general.bounds}. Hence, it remains to show that 
$s^b_{MB}(W_t) \leq 2b+3$. For this, we consider the following simple construction:
We fix the Ramsey number
$r=r_2(K_{\max\{t/2,2b+2\}})$,
we set $q=r(b+1)$ and let $\tilde{q} = 3\binom{r}{2b+2} + r + 1$.
We then let $G_0$ be a star with $q$ edges and centre vertex $v_0$.
Then for every set $S\subseteq V(G_0)$ of size $2b+3$ we add a set $A_S$ of $\tilde{q}$ new vertices and also add all possible edges between $S$ and $A_S$; we call the resulting graph $G$. As before, since $G_0$ does not contain cycles, we have that 
$G_0 \MBarrow_b W_t$ cannot hold.
Thus, it only remains to show $G\MBarrow_b W_t$. Then the claimed upper bound
$s^b_{MB}(W_t) \leq 2b+3$ follows with the same argument as in (a).

In order to show $G\MBarrow_b W_t$, consider the following simple strategy for Maker: At first, Maker claims $r$ edges of $G_0$, which is possible since $e(G_0) = r(b+1)$. Immediately afterwards, let $V_0$ denote the set of vertices which are adjacent to $v_0$ in Maker's graph. 
Now Maker ensures that for every $S\subset V_0\cup \{v_0\}$
with $v_0\in S$ and $|S|=2b+3$,
she finds a vertex $w\in A_S$ such that within 3 rounds, she claims $wv_0$ and two further edges between $S$ and $w$. She can do this by the following reason.  Running over all possible $S$, Maker plays at most $3\binom{r}{2b+2}$ further rounds. Hence, whenever Maker fixes any set $S$ as described,
she can still find a vertex $w\in A_S$ such that all edges between $w$ in $S$ are free, since $|A_S| = 3\binom{r}{2b+2} + r + 1$.
Once such a vertex $w$ is fixed,
she can first claim $wv_0$ and afterwards two further edges between $w$ and $S$, since $|S|=2b+3$.

It now remains to show that, by following this strategy, Maker manages to occupy a copy of $W_t$.
This can be done by a Ramsey-type argument. Consider an auxiliary complete graph on the vertex set $V_0$. Colour an edge $xy$ in this graph red if, for some set $S$, there is a vertex $w\in A_S$ such that Maker claims all the edges $v_0w,xw,yw$. Otherwise colour $xy$ blue. Due to Maker's strategy we know that in every subset of $V_0$ of size $2b+2$ there must be at least one red edge, and hence the auxiliary colouring does not contain a blue copy of $K_{2b+2}$.
By the definition of $r$,
we know that then the colouring must contain a red copy of $K_{t/2}$. Let $v_1,\ldots,v_{t/2}$ denote the vertices of such a copy.
Then by the definition of the colouring we know that we find distinct vertices $w_1,\ldots,w_{t/2}$ such that
for every $i\in [t/2]$, Maker claims all the edges $v_0w_i,v_iw_i,v_{i+1}w_{i}$
(where we set $v_{t/2+1}:=v_1$).
Thus, the sequence
$(v_1,w_1,v_2,w_2,\ldots,v_{t/2},w_{t/2})$ represents a cycle of length $t$ in Maker's graph such that all its vertices are connected to $v_0$ by an edge of Maker, i.e.~Maker has a copy of $W_t$.

\medskip

Finally, we prove (c).
For this, let $F$ be a perfect $(b+1)e(T)$-ary tree with the same height as $T$. Then in the first $e(T)$ rounds in the $(1:b)$ game on $F$,
Maker can always pick an arbitrary non-leaf vertex $x$ of $F$ and find a free edge between $x$ and one of its children, since $d_F(x) \geq (b+1)e(T)$. Therefore, Maker can claim a copy of $T$ in the game on $F$ within $e(T)$ rounds in a greedy-kind of way, and we get $F \MBarrow_b T$. Since every minimal MB- or BM-winnable graph contained in $F$ must be a tree, we conclude
$s_{MB}^b(T)=s_{BM}^b(T)=1$.
\end{proof}

\medskip

\section{Concluding remarks and open problems} \label{sec:concluding}

\textbf{Comparing $H$-games and Ramsey graphs.}
In this paper we defined winnable graphs
and introduced the notation $G \MBarrow_b H$ as a natural analogue of
the Ramsey arrow notation.
As pointed out by Proposition~\ref{prop:stealing}, we know that
$G\rightarrow_2 H$ always implies $G\MBarrow_1 H$.
Indeed, the assumption $G\rightarrow_2 H$ ensures that
in the $H$-game on $G$ at least one player must occupy a copy of $H$,
and hence one player must have a strategy to always claim a copy of $H$.
If this player is Maker, then she wins the game. However, if this player is Breaker, 
then Maker (as first player) can simply steal the strategy of Breaker, 
hence claim a copy of $H$ and win the game as well. 

Unfortunately, this strategy stealing argument breaks down when we consider biased games.
Yet, we believe that an analogous statement should be true for $b\geq 2$.

\begin{conjecture}
If $b\in\mathbb{N}$ and $G,H$ are graphs such that
$G\rightarrow_{b+1} H$ holds, then $G \MBarrow_b H$.
\end{conjecture}

\medskip

\textbf{General bounds on minimum degrees.}
In this paper we also introduced the parameters $s_{MB}^b(H)$ and 
$s_{BM}^b(H)$ as a game analogue of the Ramsey parameter $s_{q}(H)$.
We proved general bounds in Theorem~\ref{thm:general.bounds} 
such that in many cases the lower and upper bound are close to each other,
and we determined these parameters precisely in the cases when $H$
is e.g.~a tree, a cycle, a complete bipartite graph or a wheel with an even number of spokes.
Still, regarding our general bounds, there are some natural questions left to be answered. 
In particular, we wonder whether it can make a difference who the first player is.

\begin{problem}
Is it true that for every $b\in \mathbb{N}$ and every graph $H$, we have
$s_{MB}^b(H) = s_{BM}^b(H)$?
\end{problem}

Moreover, we wonder whether for unbiased $H$-games there exist examples $H$ such that neither of 
our general lower and upper bounds is the actual value of $s_{MB}^1(H)$.

\begin{problem}
Is there a graph $H$ such that
$2\delta(H) - 1 \neq s_{MB}^1(H) \neq 2\Delta_c(H) - 1$
for all proper vertex colourings $c$?
\end{problem}
\medskip

\textbf{Minimal winnable graphs for cliques.}
One of the most interesting graphs to be looked at, 
both in the Ramsey setting and in the positional games setting,
is the clique $K_t$. Theorem~\ref{thm:general.bounds} implies that
$s_{MB}^1(K_t)=2t-3$, while for fixed $b\geq 2$ we know that the order of
$s_{MB}^b(K_t)$ is between $bt$ and $bt\ln(t)$.
It would be interesting to close this gap.

\begin{problem}
Determine $s_{MB}^b(K_t)$ for every $b,t\in\mathbb{N}$.
\end{problem}

Also recall that Burr, Erd\H{o}s and Lov\'asz~\cite{burr1976graphs}
studied other graph parameters next to the minimum degree.
For instance, they proved that the smallest vertex-connectivity among all
minimal $2$-Ramsey graphs for $K_t$ is 3,
and that the maximum degree among all these $2$-Ramsey graphs is unbounded.
Naturally, this leads to the following questions.

\begin{problem}
Determine the smallest vertex-connectivity among all graphs in $\mathcal{M}^1_{MB}(K_t)$.
\end{problem}

\begin{problem}
Is the maximum degree among all graphs in $\mathcal{M}^1_{MB}(K_t)$ unbounded?
If not, what is the largest maximum degree among all graphs in $\mathcal{M}^1_{MB}(K_t)$?
\end{problem}

\bigskip

\textbf{Equivalent graphs for cliques.}
Let us write $G \sim H$ if $\mathcal{M}_2(H)=\mathcal{M}_2(G)$,
i.e.~if $H$ and $G$ have the same $2$-Ramsey graphs,
and call $G$ and $H$ \emph{Ramsey-equivalent} in this case.
Analogously, let us write $G \MBequiv H$ if $\mathcal{M}_{MB}^1(H)=\mathcal{M}_{MB}^1(G)$,
i.e.~if $H$ and $G$ have the same winnable graphs for bias $1$.

Due to a result of Fox et al.~\cite{fox2014ramsey}
it is known that the only connected graph that is Ramsey equivalent with the clique $K_t$
is the clique $K_t$ itself; see also~\cite{clemens2020minimal} for a generalization to more than 2 colours.
Moreover, due to their results, the authors in~\cite{clemens2020minimal,fox2014ramsey} 
asked whether there exist non-isomorphic connected graphs $H_1$ and $H_2$ that are Ramsey-equivalent. The problem is still unsolved.
But in~\cite{clemens2020minimal}
it is shown that the answer is no if "connected" is replaced with "3-connected",
and in~\cite{axenovich2017conditions,savery2022chromatic} further conditions for non-equivalence are proven.
In contrast to this, the analogous question for Maker-Breaker games can be answered in the affirmative
by a simple argument, which we include here. Let $K_t\cdot K_2$ be the graph obtained by attaching an edge to some vertex of $K_t$.

\begin{thm}\label{thm:equivalence}
For all integers $b\geq 1$ and $t\geq 3$, we have $\mathcal{M}_{MB}^b(K_t) = \mathcal{M}_{MB}^b(K_t\cdot K_2)$.
\end{thm}

\begin{proof}
Let $G$ be a graph. $G \MBarrow_b K_t\cdot K_2$ implies $G \MBarrow_b K_t$, since $K_t$ is a subgraph of $K_t\cdot K_2$. It remains to show that if $G$ is \emph{MB-winnable} for $K_t$, then it is also \emph{MB-winnable} for $K_t \cdot K_2$.

\smallskip

Assume $G$ is \emph{MB-winnable} for $K_t$ with bias $b$, meaning that Maker has a strategy to win the $K_t$-game on $G$. During that strategy, there must be a round, where Maker creates $b + 1$ immediate threats, so that she can win in the next round. Assume Maker creates those threats in some round by claiming the edge $xy$, i.e.~after this move there are $b+1$ distinct copies of $K_t^-$ (i.e.~clique minus one edge) in Maker's graph, such that all of them contain the edge $xy$ and such that the $b+1$ edges which would complete a $K_t$ are still free. Then Breaker cannot block all of them and Maker can complete one of those copies to a $K_t$ in the next round.

Afterwards, let $v$ be a vertex that is not part of that copy of $K_t$, but that was part of one of the copies of $K_t^-$ Maker used as a threat. Then at least one of the edges $vx$ and $vy$ belongs to Maker's graph, completing a copy of $K_t \cdot K_2$.
\end{proof}

Naturally, this result leads to the following problem.

\begin{problem}
Given $t\in \mathbb{N}$. Characterize all connected graphs $H$
such that $H \MBequiv K_t$.
\end{problem}

Moreover, Fox et al.~\cite{fox2014ramsey} showed that $K_t \not\sim K_t + 2K_{t-1}$,
and later, Bloom and Liebenau~\cite{bloom2018ramsey} proved that $K_t \sim K_t + K_{t-1}$
for every $t\geq 4$.  Motivated by these results, we can ask for the following game analogue.

\begin{problem}
Given $s,t\in \mathbb{N}$ with $s < t$. Determine asymptotically the largest value $k=k(s,t)$
such that $K_t \MBequiv K_t + k K_s$.
\end{problem}

\medskip

\textbf{A game analogue of size-Ramsey numbers.}
Another interesting parameter in Ramsey theory is the \emph{size-Ramsey number},
as introduced by Erd\H{o}s et al.~\cite{erdos1978size}, and defined as
$$
\hat{r}_2(H) := \min\{ e(G):~ G\rightarrow_2 H\} .
$$
Trivially, it holds that $\hat{r}_2(H) \leq \binom{r_2(H)}{2}$,
and this bound is tight for cliques as shown in~\cite{erdos1978size}.
On the other end, it was shown by Beck~\cite{beck1983size}
that the size-Ramsey number of a path $P_n$ grows linearly
with the number $n$ of vertices, resolving a question of
Erd\H{o}s~\cite{erdos1981combinatorial}. The result of Beck has been improved
a couple of times, with the current best bounds being
$(3.75-o(1)) n \leq \hat{r}_2(P_n) \leq 74n$,
proven by Bal and DeBiasio~\cite{bal2022new}, and by Dudek and Pra{\l}at~\cite{dudek2015alternative}.
Turning to games, we can now define an analogue parameter
$$
\hat{r}_{MB}(H) := \min\{ e(G):~ G\MBarrow_1 H\} .
$$
The simple construction in Figure~\ref{fig:size.path.construction} can be used to show that $\hat{r}_{MB}(P_n) \leq 4n-11$ for $n\geq 4$,
but we do not know whether this is asymptotically best possible.

\begin{center}
\begin{figure}[t] 
	\begin{center}
\includegraphics[scale=0.8]{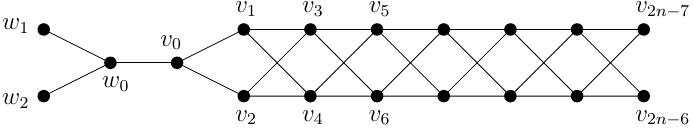}
	\end{center}
	\caption{MB-winnable graph for $P_n$. Strategy: 
    Maker first claims $v_0w_0$. Afterwards, she pairs
    edges that start in the same vertex $v_i$ or $w_i$ and go to vertices of higher index.}
	\label{fig:size.path.construction}
\end{figure}
\end{center}

\begin{problem}
Determine $\hat{r}_{MB}(P_n)$.
\end{problem}

Obviously, the same question can be asked for other graphs such as cycles and wheels. Next to this, we note that Beck~\cite{beck1990size} later asked whether for graphs of bounded maximum degree the size-Ramsey number grows linearly in the number of vertices. This was answered negatively by Rödl and Szemer\'edi~\cite{rodl2000size}, who proved that there exist graphs $H_n$ on $n$ vertices and with maximum degree $3$ such that $\hat{r}(H_n) = \Omega(n\log^{1/60}n)$, and this bound was improved further in~\cite{tikhomirov2024bounded}. However, the analogous Maker-Breaker version of Beck's question was answered positively by Gebauer~\cite{gebauer2013size}.

\medskip

\bibliographystyle{amsplain}
\bibliography{references}

\end{document}